\documentclass[11pt]{lmcs}

\usepackage{double-data-macros}

\theoremstyle{definition}\newtheorem{con}[thm]{Construction}
\theoremstyle{definition}\newtheorem{dis}[thm]{Discussion}

\usepackage{hyperref}
\usepackage[capitalize,noabbrev,nameinlink]{cleveref}
\crefname{con}{Construction}{Construction}
\crefname{dis}{Discussion}{Discussion}
\makeatletter
\AddToHook{env/thm/begin}{\crefalias{thm}{theorem}}
\AddToHook{env/lem/begin}{\crefalias{thm}{lemma}}
\AddToHook{env/cor/begin}{\crefalias{thm}{corollary}}
\AddToHook{env/prop/begin}{\crefalias{thm}{proposition}}
\AddToHook{env/defi/begin}{\crefalias{thm}{definition}}
\AddToHook{env/rem/begin}{\crefalias{thm}{remark}}
\AddToHook{env/exa/begin}{\crefalias{thm}{example}}
\AddToHook{env/con/begin}{\crefalias{thm}{con}}
\AddToHook{env/dis/begin}{\crefalias{thm}{dis}}
\makeatother

\title{Quantification in Double-Categorical Database Schemas}
\author{Michael Lambert}
\address{Norwich University, 158 Harmon Drive, Northfield VT, 05663}
\email{michael.james.lambert@gmail.com, mlamber2@norwich.edu}

\begin{document}
\begin{abstract}
  Double-categorical database schemas are enriched with universal quantification 
  in the form of right adjoints to substitution. This allows phrasing of the 
  important query, \emph{relational division}. It is shown that such right 
  adjoints together with suitable tabulators interpret modal operators, provide 
  cartesian closed structure, and, when combined with global cocartesian 
  structure, interpret first-order predicate logic and thus description logic. 
  These structures are applied throughout to querying and optimization. It is 
  seen, for example, that Frobenius reciprocity and Beck-Chevalley hold in any 
  suitably structured double database schema and that these provide pushdown 
  optimization rules. Such optimization rules are thus provably error-free as a 
  property of the schema. Likewise, negation queries are derived from 
  cocartesian and local implication structure. 
\end{abstract}

\maketitle

\section{Introduction}

The double-categorical approach to database schematization and knowledge 
representation (KR) began in \cite{lambert2025} is in this paper extended by the 
introduction of quantification as a further operation on double ologs. Such 
double ologs are a synthesis of the functional \cite{Spivak2010,spivak2012} and 
relational \cite{patterson2017} approaches to data schemas within the single 
framework of a suitably structured cartesian equipment known as a `double 
category of relations' \cite{lambert2022}. A data instance on such a double olog 
is a cartesian double functor valued in relations. Standard relational database 
query operations such as select, filter and join are all native operations given 
by the equipment and cartesian structure of the given double olog; instances 
preserve all such structure and compute the query results directly in relations. 
This has the effect of bypassing the foregoing standard of lifting to induced 
external adjunctions between functor categories to perform functorial data 
migration \cite{Spivak2010}. Up until the present work, however, more advanced 
queries such as \emph{relational division} \cite{codd1972} were not known to be 
possible within this double-categorical framework.

To perform such queries, the present work studies the effect of asking for right 
adjoints to the substitution functors coming from the equipment structure in any 
double olog. These right adjoints are \emph{dependent products} or \emph{fibered 
products} in the language of \emph{fibrational semantics} 
\cite{lawvere1970, jacobs1999}. These dependent products have appeared in purely 
bicategorical contexts as \emph{Kan quantification} \cite{lawvere2002}. When 
combined in the present situation with suitable restrictions and tabulators they 
give (1) division operators, (2) modal operators, and (3) exponentials. Each of 
these, the queries they allow, and their interrelations with the standing 
modular laws, Frobenius reciprocity and Beck-Chevalley condition, are studied 
here in detail. The culminating result is that a double olog with dependent 
products, strong tabulators, and coproducts satisfying a global distributivity 
law is locally a model of first-order predicate logic enhanced with modal 
operators modelling possibility (liveness) and necessity (safety). Thus, such 
ologs also interpret \emph{description logic} \cite{baader2010} which forms the 
basis of OWL and web3. In this way, the present paper advances the project of 
seeing structured equipments as a models of logics over type theories 
\cite{jacobs1999}.

The theoretical narrative of this development is itself of interest. The 
external Beck-Chevalley condition is shown to hold in any `double category of 
relations' with strong tabulators (\cref{lemma:beck-chevalley}). This results 
from the theoretical advance of \cite{hoshino2025} showing that the condition of 
\emph{unit-purity} is automatically implied by the assumption of 
\emph{discreteness} \cite[Definition 2.1]{carboni1987} in any `double category 
of relations'. Formerly \cite{lambert2022}, unit-purity figured as a 
facilitating but \emph{ad hoc} assumption that is now seen to be redundant. 
Likewise, we show that Frobenius reciprocity holds in any `double category of 
relations' (\cref{prop:Frobenius-reciprocity}). This was asserted in 
\cite{lambert2022} and proved in appropriate references. It is seen here to be a 
consequence of the modular laws stemming from the fact that any `double category 
of relations' is compact closed with a trivial duality involution 
\cite[Theorem 2.4]{carboni1987}. Frobenius reciprocity is used to show that any 
double olog with dependent products and strong tabulators is locally cartesian 
closed in the sense that every hom category is cartesian closed 
(\cref{theo:locally-cartesian-closed}). The crucial result enabling this 
(\cref{lemma:identity-proarrow-exponentiable}) shows that identity proarrows are 
always exponentiable in any `double category of relations'. Essentially, the 
discreteness axiom forces the identity proarrows to act as very strict identity 
predicates akin to genuine diagonals. We thus recover formulas analogous to 
those of \cite[Theorem 10.5.4]{jacobs1999} without the need of elaborate 
comprehension structures. Nor do we require that the base category is cartesian 
closed. This establishes a direct connection between dependent (co)products, 
tabulators, and exponents. Finally, it is shown that global cocartesian 
structure with a global distributive law localizes to give local coproducts over 
which local products distribute. In this sense, a main result is that 
distributivity localizes (\cref{theorem:local-distributivity}). 

Several difficult queries in relational database theory are captured in this 
framework. The first is the relational division operation originating in 
\cite{codd1972}. The classical phrasing of this operation has a notoriously 
difficult syntax and is a major source of computational bottlenecking in 
implementations. Here it is handled by precomposing a well-chosen dependent 
product with a suitable restriction functor. The inevitability of this 
development is showcased by the first three sections 
(\cref{section:relational-division,section-ologs,section:varieties-of-quantification}) 
which have a discursive pacing and recount 
the narrative of the present discovery of the solution. Of course, inasmuch as 
right adjoints to substitution are already known to encode division (hence the 
very notion of a \emph{division allegory} \cite{freyd1990} and also 
\cite[\S 11]{lambert2022}, this solution is in hindsight probably the obvious 
one, but to our knowledge the connection has not been formalized for categorical 
database querying. Likewise, negation queries are captured using cartesian 
closed structure and local coproducts. These are combined with existing joins to 
do otherwise difficult compound queries with relative ease 
(\cref{section:cocartesian-negation}). Modal queries involving possibility and 
necessity operators are also introduced as applications of dependent products 
and tabulators (\cref{section:modality}). These bear a superficial resemblance 
to division queries but are indeed closer to liveness and safety properties 
\cite{lamport1977,alpern1985,alpern1987}. Disjunction queries are studied for 
the first time under the guise of \emph{local coproducts} which are induced from 
global cocartesian structure (\cref{section:cocartesian-negation}).

A major application of the present work is to query optimization rules. 
Certainly, join is generally regarded as the most costly of standard queries; 
and select is essentially free. This is borne out for us by the fact that join 
is typically a complicated restriction along diagonals and select is achieved by 
a simple projection from a factor \cite{lambert2025}. Filter and collapse are 
closer overall. Filter is executed by restricting along a value or subtype; 
whereas collapse is best thought of as \emph{extraction of metadata} or 
\emph{abstraction}, that is, basically, extension along a surjection. Inasmuch 
as the latter is something like passing to equivalence classes, or a taking a 
quotient, or a colimit, we regard it as generally more expensive than filter, 
but still less than join. We adopt a heuristic hierarchy of relative query 
expense: 
  \begin{equation*}\label{equation:query-expense}
    \text{select column } \lessapprox  \text{ filter by value } \lessapprox \text{ collapse/abstract } \lessapprox \text{ join/conjunction}.
  \end{equation*}
Relative to this hierarchy, we see both the Beck-Chevalley condition and 
Frobenius reciprocity as rewrite and optimization rules for certain compound 
queries. We show 
(\cref{remark:Beck-Chevalley-Optimization,example:beck-chevalley}) that the 
former is a \emph{filter-pushdown} optimization that takes a collapse followed 
by filter and rewrites it as a more efficient filter followed by a collapse. In 
our illustrating example, the more efficient query ends up being a filter 
followed by a select column. Frobenius reciprocity we show 
(\cref{remark:Frobenius-Optimization,example:frobenius}) is 
a \emph{join-pushdown} which rewrites a collapse following an expensive join/
conjunction as a join following a collapse. This does not eliminate the need of 
a join but the collapse/abstract first reduces the size of the dataset on which 
it is performed. The point of deriving Beck-Chevalley and Frobenius reciprocity 
in our ologs is that these rewrite and optimization rules are thus native 
features of the database schemas and preserved by the functorial and fibrational 
semantics of any data instance. Such optimization is thus in this sense 
baked-in. As a result, optimization type-checks automatically and is provably 
error-free without the need of \emph{ad hoc} or heuristic case-by-case rules.

This project began several years ago with the work leading up to 
\cite{lambert2025} where it was noticed that many of the examples of migration 
in \cite{Spivak2010} were queries that could be realized by pullbacks and 
projections. Double categories had already entered our thinking on KR as a 
result of the desirability of a general framework where genuinely relational 
aspects attained a first-class status. And in this context, it was realized that 
cartesian equipments carried exactly the derived structure to model natively 
select, filter and join as operations internal to a double olog generated by a 
presentation. Both ingredients, namely having products and being a fibration, 
are indispensable. Others such as the hypothesis 
of being a `double category of relations' or having tabulators are highly 
convenient inasmuch as (for example) the derived property of \emph{compactness} 
seems to solve a difficult bookkeeping problem introduced by the sidedness of 
relations. Of the two necessary ingredients, the equipment structure, that is, 
the fact that the source-target projection is a fibration, is the subtle and 
rich one. The use of fibrations to model databases is not exactly new 
(e.g. \cite{rosebrugh1992,islam1994}). By happy coincidence the power of a 
fibrational perspective was rediscovered in the present and foregoing work on 
streamlining the account of data querying. Genuinely novel is the introduction 
of the technical machinery of structured double categories presenting and 
formalizing ologs; and its combination with the idea that instancing via both 
double-categorical functorial and fibrational semantics amounts to database 
querying. This paper is an expos\'e on the consequences especially of the latter 
aspect.

To conclude this introduction, a summary of the main results of the paper is as 
follows:
  \begin{enumerate}
    \item Relational division is recast as a composite operation in a suitably structured double category using dependent products (\cref{example:recast-relational-division}).
    \item Beck-Chevalley and Frobenius reciprocity are realized as query optimization rewrite rules (\cref{remark:Beck-Chevalley-Optimization}, \cref{remark:Frobenius-Optimization}).
    \item Dependent products and strong tabulators are seen to give local cartesian closure and a formula for exponentials (\cref{theo:locally-cartesian-closed}).
    \item Further cocartesian structure allows negation queries (\cref{example:negation}).
    \item Distributivity localizes (\cref{theorem:local-distributivity}) without the assumption of cartesian closure or even the existence of dependent products.
    \item Any suitably structured double olog, viewed as an equipment, interprets first-order predicate logic, hence all the traditional operations of relational algebra (\cref{corollary:first-order-fibration}).
    \item What we call a \emph{FOML double olog} interprets also modal necessity and possibility operators and thus gives a semantics of \emph{description logic} (\cref{definition:interpret-description-logic}).
  \end{enumerate}
Double-categorical conventions are consistent with the previous 
\cite{lambert2024a,lambert2025}. Our \cite{lambert2022} on `double categories of 
relations' is taken for granted especially for the crucial notion of 
\emph{discreteness}. Likewise \cite{carboni1987} is an indispensable reference 
for the bicategorical description of \emph{compactness}. \cite{jacobs1999} is a 
guiding resource for fibrational interpretations of logics over type theories. 
As with our previous work on double ologs, examples of data instances are taken 
from a non-standard source. In this instance they are taken from the fantasy 
RPG, The Elder Scrolls V: Skyrim.

\tableofcontents

\part{Dependent Products \& Relational Division}

This first part of the paper is devoted entirely to setting up, explaining, and 
recasting in appropriate double-categorical machinery the queries we are 
interested in. The main motivating query is again \emph{relational division} 
introduced in \cite{codd1972} essentially as a \emph{check inventory against 
list} query. The first three sections 
(\cref{section:relational-division,section-ologs,section:varieties-of-quantification}) 
are a discussion-style narrative recounting the original syntax of relational 
division and showing in summary a form of the analysis that led to the 
introduction of dependent products and the formulation of the solution as 
presented in \cref{example:recast-relational-division}. Dependent coproducts and 
products are realized as species of \emph{Kan quantification} 
\cite{lawvere2002}. We study Beck-Chevalley in this connection as a rewrite and 
optimization rule (\cref{section:existential-queries}). On this basis, 
universally-quantified queries are introduced using dependent products in 
\cref{section:universal-queries} and leading to the formulation of relation 
division. ologs with this kind of structure are axiomatized as $\prod$-double 
ologs. Along the way, we survey 6 different kinds of queries 
(\cref{fig:quant-queries}), each of which is derived from quantification over a 
relation in some form. Two of these are modal queries: one possibility and the 
other necessity. The latter bears a superficial resemblance to division but is 
better captured using suitably structured tabulators (\cref{section:modality}).

\section{Relational Division}
\label{section:relational-division}
Relational division is a query introduced in \cite[\S 2]{codd1972} that behaves 
essentially as a universal quantifier over given specified values in the column 
of a table. First we illustrate this operation with a generic example. We work 
with binary relations $R \colon C\proto A$ which are of course equivalently 
monic arrows $\langle d,c\rangle \colon R\rightarrowtail C\times A$. Recall the 
notation from the reference
  \begin{equation*}
    r[A] = c(r) \text{ where } r\in R
  \end{equation*}
and likewise for $r[C] = d(r)$. Likewise, for any binary relation $T$, let 
$g_T(x) = \lbrace y\mid xTy\rbrace$, that is, the set of all $y$ in the target 
of $T$ related to $x$ under $T$. Call this the \textbf{$T$-orbit} of $x$. Now, 
for the example, consider the two tables:
  \begin{equation*}
    \begin{tabular}{| l | l | }
        \hline\multicolumn{2}{| c |}{\bf R}\\
        \hline {\bf C }&{\bf A}\\
        \hline 1 & a \\
        \hline 1 & b \\
        \hline 1 & c \\
        \hline 2 & a \\
        \hline 2 & b \\
        \hline 3 & b \\
        \hline
    \end{tabular} \qquad\qquad\qquad
    \begin{tabular}{| l | l | }
        \hline\multicolumn{2}{| c |}{\bf S}\\
        \hline {\bf B }&{\bf D}\\
        \hline a & 2 \\
        \hline b & 1 \\
        \hline b & 2 \\
        \hline a & 3 \\
        \hline
    \end{tabular}.
  \end{equation*}
Note that $A$ and $B$ have entries from the same set $\lbrace a,b,c\rbrace$. 
This is required to perform division. Now, in the notation of the reference, a 
division query is 
    \begin{equation*}
      R[A\div B]S = \lbrace r[C]\mid r\in R \text{ and } S[B]\subset g_R(r[C])\rbrace = \lbrace 1,2\rbrace.
    \end{equation*}
We unpack this carefully to illustrate the operation. The possible values of 
$r[C]$ are the distinct entries of the $C$-column of $R$, namely, $1$, $2$, and 
$3$. Likewise, $S[B] = \lbrace a,b\rbrace$ is the set of all values of the 
$B$-column of $S$. So, there are three sets of the form $g_R(r[C])$, only two of 
which contain $S[B]$, as in: 
    \begin{align*}
      S[B] &= \lbrace a,b\rbrace \nsubseteq \lbrace b\rbrace = g_R(3) \\
      S[B] &= \lbrace a,b\rbrace \subset \lbrace a,b,c\rbrace = g_R(1)\\
      S[B] &= \lbrace a,b\rbrace = \lbrace a,b\rbrace = g_R(2).
    \end{align*}
Hence the division query returns $\lbrace 1,2\rbrace$, exactly \emph{the set of 
all $C$-values whose orbit contains all the values of the $B$-column of $S$}. 
More prosaically, a $C$-value is selected by the query iff it is related to 
everything on the list of $B$-values. Thus, \emph{universal quantification} 
enters as an essential aspect of the query. Now, in the containment relations 
above, the universally quantified set does not need to completely describe the 
orbit of any given element satisfying the division query; that is, $1$ is 
related to $c$, but this is not a $B$-value. Likewise $3$ is related to 
something on the list, but not everything.

A more concrete example illustrating the same query is the following. Consider 
the two relational tables below. On the left, rows pair an apothecary with an 
ingredient sold there; on the right, rows pair an ingredient with an effect 
produced by the combination of that ingredient with another in a potion. 
    \[
    \begin{tabular}{| l | l | }
        \hline\multicolumn{2}{| c |}{\bf Inventory}\\
        \hline {\bf Apothecary }&{\bf Ingredient}\\
        \hline Arcadia's Cauldron & giant's toe \\
        \hline Arcadia's Cauldron & wheat \\
        \hline Arcadia's Cauldron & void salts \\
        \hline Elgrim's Elixers & giant's toe \\
        \hline Elgrim's Elixers & wheat \\
        \hline The White Phial & wheat \\
        \hline
    \end{tabular} \qquad\qquad
    \begin{tabular}{| l | l | }
        \hline\multicolumn{2}{| c |}{\bf Potion}\\
        \hline {\bf Ingredient }&{\bf Effect}\\
        \hline giant's toe & fortify health \\
        \hline wheat & fortify health \\
        \hline wheat & lingering damage magicka \\
        \hline giant's toe & damage stamina regen \\
        \hline
    \end{tabular}
  \]
Note that the second table does not list the other ingredients required to 
create the potion producing the effect. Our expertise is that any two such 
ingredients with the same effect will create the desired potion when combined. 
Now, suppose we want to create a fortify health potion but are to lazy to fight 
a giant to get his toe. We would also really prefer to one-stop-shop (maybe the 
merchant will give us a deal if we buy in bulk). The orbits are then the 
inventories:
    \begin{align*}
      g_{\text{Inventory}}(\text{Arcadia's}) &= \lbrace \text{giant's toe, wheat, void salts}\rbrace \\
      g_{\text{Inventory}}(\text{Elgrim's}) &= \lbrace \text{giant's toe, wheat}\rbrace \\
      g_{\text{Inventory}}(\text{White Phial}) &= \lbrace \text{wheat} \rbrace
    \end{align*}
The the set of values quantified over is our shopping list: $\lbrace 
\text{giant's toe, wheat}\rbrace$. In the notation above, the division query is 
then 
    \begin{align*}
      &\text{Inventory}[\text{Ingredient}\div \text{Ingredient}]\text{Potion} \\
      =\;&\lbrace \text{Apothecary} \mid \lbrace \text{giant's toe, wheat}\rbrace\subset g_\text{Inventory}(\text{Apothecary})\rbrace
    \end{align*}
which is just to say that we are looking for the proprietors who have all the 
ingredients on our shopping list. The result is of course $\lbrace 
\text{Arcadia's, Elgrims}\rbrace$. It does not matter that Arcadia's also has 
void salts; but it does matter that The White Phial only has wheat but no 
giant's toe.

\begin{rem}
  There is something awkward about these examples inasmuch as we would probably 
  more likely have tables better reflecting our alchemical expertise such as 
      \[
        \begin{tabular}{| l | l | }
            \hline\multicolumn{2}{| c |}{\bf Apothecary}\\
            \hline {\bf Name }&{\bf Ingredient}\\
            \hline Arcadia's Cauldron & giant's toe \\
            \hline Arcadia's Cauldron & wheat \\
            \hline Arcadia's Cauldron & void salts \\
            \hline Elgrim's Elixers & giant's toe \\
            \hline Elgrim's Elixers & wheat \\
            \hline The White Phial & wheat \\
            \hline
        \end{tabular} \qquad
        \begin{tabular}{| l | l | }
            \hline\multicolumn{2}{| c |}{\bf Potion}\\
            \hline {\bf Ingredient }&{\bf Effect}\\
            \hline giant's toe & fortify health \\
            \hline wheat & fortify health \\
            \hline wheat & lingering damage magicka \\
            \hline giant's toe & damage stamina regen \\
            \hline swamp fungal pod & restore health \\
            \hline void essence & restore health\\
            \hline
        \end{tabular}.
      \]
  However, division with quantification over the whole the Ingredient column in 
  the Potion table now produces the empty set since no merchant has all those 
  items on hand. And \cite[\S 2.4]{codd1972} recognizes this in the example 
  queries. Namely, example query \#8 to find suppliers supplying at least parts 
  12, 13 and 15 restricts the part \# values to only 12, 13 and 15 without 
  ranging over the whole column before executing the division query. This is 
  exactly a \emph{shopping list} kind of example. In other words, technically, 
  the division query is not set up to do selection from a sub-list of a column 
  on its own. A restriction to just the values of interest from the second table 
  is employed implicitly before doing the division query. In the present example 
  immediately above, to execute our query, we would filter the rows with the 
  value ``fortify health'' first and then do the division query with the 
  ingredients column. The flexibility of the division query as written is that 
  it can be applied to any two tables with columns having the same types of 
  values; the expense is that it might not return anything and may require 
  restrictions to produce results in intuitive cases.
\end{rem}

\section{Ologs for Division}
\label{section-ologs}

Having reviewed relational division, we now recast the situation from a 
double-categorical olog-perspective \cite{lambert2025}. The data is schematized 
by two basic relations
    \begin{equation*}
      \begin{tikzcd}
        {\fbox{Apothecary}} && {\fbox{Ingredient}}
        \arrow["{\text{Inventory}}"{inner sep=.8ex}, "\shortmid"{marking}, from=1-1, to=1-3]
      \end{tikzcd}\qquad\qquad
      \begin{tikzcd}
        {\fbox{Ingredient}} && {\fbox{Effect}}
        \arrow["{\text{Potion}}"{inner sep=.8ex}, "\shortmid"{marking}, from=1-1, to=1-3]
      \end{tikzcd}.
    \end{equation*}
Suppose that these relations are instanced by the two tables at the end of the 
previous section, thought of as living in $\Rel$. Now, the olog generated will 
be a cartesian equipment. So, given an individual as a global element of a type, 
say, $\text{Arcadia's}\colon 1\to \fbox{Name}$, we can form the restriction
    \begin{equation*}
      \begin{tikzcd}
        1 &&& {\fbox{Ingredient}} \\
        {\fbox{Apothecary}} &&& {\fbox{Ingredient}}
        \arrow[""{name=0, anchor=center, inner sep=0}, "{\text{Arcadia's Inventory}}"{inner sep=.8ex}, "\shortmid"{marking}, from=1-1, to=1-4]
        \arrow["{\text{Arcadia's}}"', from=1-1, to=2-1]
        \arrow[from=1-4, to=2-4]
        \arrow[""{name=1, anchor=center, inner sep=0}, "{\text{Inventory}}"'{inner sep=.8ex}, "\shortmid"{marking}, from=2-1, to=2-4]
        \arrow["{\text{res}}"{description}, draw=none, from=0, to=1]
      \end{tikzcd}
    \end{equation*}
which in the instance is precisely the named merchant's inventory. This would 
play the role of $g_R(x)$ in the previous section. Likewise, our ingredient list 
for a fortify health potion is schematized by a restriction:
    
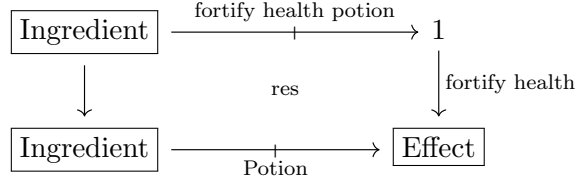
\begin{equation*}
      \begin{tikzcd}
        {\fbox{Ingredient}} &&& 1 \\
        {\fbox{Ingredient}} &&& {\fbox{Effect}}
        \arrow[""{name=0, anchor=center, inner sep=0}, "{\text{fortify health potion}}"{inner sep=.8ex}, "\shortmid"{marking}, from=1-1, to=1-4]
        \arrow[from=1-1, to=2-1]
        \arrow["{\text{fortify health}}", from=1-4, to=2-4]
        \arrow[""{name=1, anchor=center, inner sep=0}, "{\text{Potion}}"'{inner sep=.8ex}, "\shortmid"{marking}, from=2-1, to=2-4]
        \arrow["{\text{res}}"{description}, draw=none, from=0, to=1]
      \end{tikzcd}
    \end{equation*}
These are both examples of \emph{filter queries} discussed in 
\cite{lambert2025}. We are trying ultimately to compare the two restrictions, to 
see if the ingredients for the fortify health potion are in the various 
inventories. In the examples above, the composites 
    
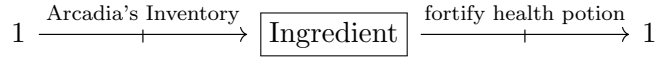
\begin{equation*}
      \begin{tikzcd}
        1 &&& {\fbox{Ingredient}} &&& 1
        \arrow["{\text{Arcadia's Inventory}}"{inner sep=.8ex}, "\shortmid"{marking}, from=1-1, to=1-4]
        \arrow["{\text{fortify health potion}}"{inner sep=.8ex}, "\shortmid"{marking}, from=1-4, to=1-7]
      \end{tikzcd}
    \end{equation*}
and 
    
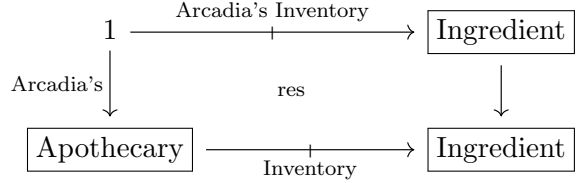
\begin{equation*}
      \begin{tikzcd}
        1 &&& {\fbox{Ingredient}} \\
        {\fbox{Apothecary}} &&& {\fbox{Ingredient}}
        \arrow[""{name=0, anchor=center, inner sep=0}, "{\text{Arcadia's Inventory}}"{inner sep=.8ex}, "\shortmid"{marking}, from=1-1, to=1-4]
        \arrow["{\text{Arcadia's}}"', from=1-1, to=2-1]
        \arrow[from=1-4, to=2-4]
        \arrow[""{name=1, anchor=center, inner sep=0}, "{\text{Inventory}}"'{inner sep=.8ex}, "\shortmid"{marking}, from=2-1, to=2-4]
        \arrow["{\text{res}}"{description}, draw=none, from=0, to=1]
      \end{tikzcd}
    \end{equation*}
are natural to study but neither returns what we want. In the first case, the 
$\Rel$-composite will return whichever ingredient Arcadia's has on the list; in 
the second case, the $\Rel$-composite returns whichever merchants have an 
ingredient on the list. This is a result of the fact that existential 
quantification over the instance of the middle type is built into composition in 
$\Rel$. Neither of course is right: the first considers only one merchant; the 
second only whether a merchant has an ingredient on the list.

The better way is to view our shopping list as a subtype of the type of 
ingredients:
    
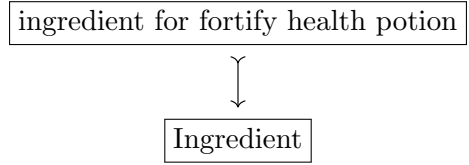
\begin{equation*}
      \begin{tikzcd}
        {\fbox{\text{ingredient for fortify health potion}}} \\
        {\fbox{Ingredient}}
        \arrow[tail, from=1-1, to=2-1]
      \end{tikzcd}
    \end{equation*}
This need not necessarily be monic. But monic arrows are perfectly well 
axiomatizable in our framework. Perhaps we use the notation to indicate we are 
thinking of it as a subtype obtained by comprehension \cite[\S 9.3]{jacobs1999}. 
Nor does it matter how the subtype comes about at this stage; it could have been 
obtained as a tabulator, or simply required as a primitive in the setup of the 
olog. Notice that this in any case is a genuine subtype that is not necessarily 
a global element (as was the case with all the examples in \cite{lambert2025}.

What matters is comparing this subtype with the inventories of all the merchants 
simultaneously and finding a way to quantify over those inventories. The 
individual inventories are implicit in the instancing of the Inventory relation, 
so we study the corner  
    
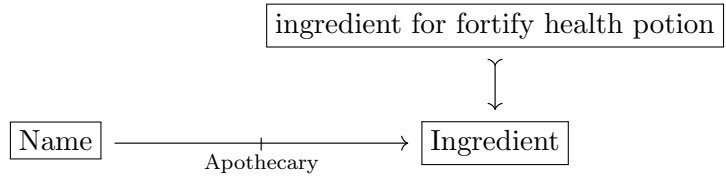
\begin{equation*}
      \begin{tikzcd}
        && {\fbox{\text{ingredient for fortify health potion}}} \\
        {\fbox{Name}} && {\fbox{Ingredient}}
        \arrow[tail, from=1-3, to=2-3]
        \arrow["{\text{Apothecary}}"'{inner sep=.8ex}, "\shortmid"{marking}, from=2-1, to=2-3]
      \end{tikzcd}.
    \end{equation*}
It is tempting just to restrict with an identity morphism on the left. This 
results in a cartesian cell 
    
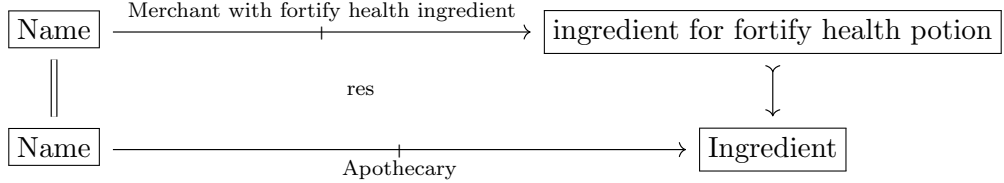
\begin{equation*}
      \begin{tikzcd}
        {\fbox{Name}} &&&&&& {\fbox{\text{ingredient for fortify health potion}}} \\
        {\fbox{Name}} &&&&&& {\fbox{Ingredient}}
        \arrow[""{name=0, anchor=center, inner sep=0}, "{\text{Merchant with fortify health ingredient}}"{inner sep=.8ex}, "\shortmid"{marking}, from=1-1, to=1-7]
        \arrow[equals, from=1-1, to=2-1]
        \arrow[tail, from=1-7, to=2-7]
        \arrow[""{name=1, anchor=center, inner sep=0}, "{\text{Apothecary}}"'{inner sep=.8ex}, "\shortmid"{marking}, from=2-1, to=2-7]
        \arrow["{\text{res}}"{description}, draw=none, from=0, to=1]
      \end{tikzcd}
    \end{equation*}
which filters the table for those rows having one of the two ingredients on the 
list:
    \begin{equation*}
      \begin{tabular}{| l | l | }
        \hline\multicolumn{2}{| c |}{\bf Apothecary}\\
        \hline {\bf Name }&{\bf Ingredient}\\
        \hline Arcadia's Cauldron & giant's toe \\
        \hline Arcadia's Cauldron & wheat \\
        \hline Elgrim's Elixers & giant's toe \\
        \hline Elgrim's Elixers & wheat \\
        \hline The White Phial & wheat \\
        \hline
      \end{tabular}.
    \end{equation*}
Notice that this gets rid of the unnecessary entry with the value ``void salts'' 
but it does not establish our query. Another operation with which our olog comes 
equipped is extension. In $\Rel$ this acts as an image, and formalizes a select 
query \cite{lambert2025}. In the present instance we might try an extension
    
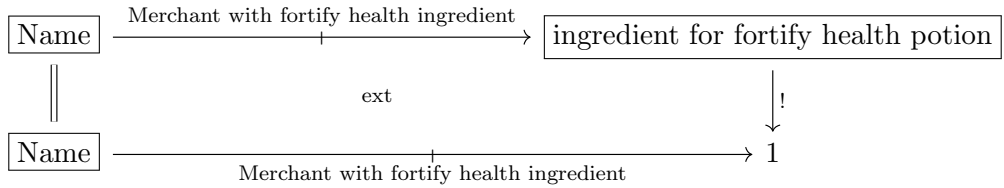
\begin{equation*}
      \begin{tikzcd}
        {\fbox{Name}} &&&&&& {\fbox{\text{ingredient for fortify health potion}}} \\
        {\fbox{Name}} &&&&&& 1
        \arrow[""{name=0, anchor=center, inner sep=0}, "{\text{Merchant with fortify health ingredient}}"{inner sep=.8ex}, "\shortmid"{marking}, from=1-1, to=1-7]
        \arrow[equals, from=1-1, to=2-1]
        \arrow["{!}", from=1-7, to=2-7]
        \arrow[""{name=1, anchor=center, inner sep=0}, "{\text{Merchant with fortify health ingredient}}"'{inner sep=.8ex}, "\shortmid"{marking}, from=2-1, to=2-7]
        \arrow["{\text{ext}}"{description}, draw=none, from=0, to=1]
      \end{tikzcd}
    \end{equation*}
which in the data instance is effectively just a list of the three merchants, 
since each has at least one ingredient. But if we went to the White Phial, we 
would be out of luck. This is owing to the fact that in $\Rel$ the extension is 
an image computed via \emph{existential quantification}. So, the composite query 
here of restriction along the subtype, followed by extension has performed an 
\emph{existential query}, returning each merchant whose inventory contains an 
ingredient on the list.

Now, this is clearly not the result we want, but it is not a dead-end, for it 
suggests a path forward. Extensions are left adjoints to restriction. Inasmuch 
as extension is thought of as existential quantification and restriction is 
thought of as substitution, we have an analogy of adjunctions 
    \begin{equation*}
      \text{ext}\dashv \text{res} \qquad\leftrightarrow \qquad \exists \dashv (-)^*.
    \end{equation*}
Of course in nice logical situations substitution also has a right adjoint, 
which is \emph{universal quantification}. This is of course exactly the kind of 
operation we are seeking to formalize. What we want then is to complete the 
adjoint analogy
    \begin{equation*}
      \text{ext}\dashv \text{res} \dashv \,\,??? \qquad\leftrightarrow \qquad \exists \dashv (-)^*\dashv \forall
    \end{equation*}
in the theater of ologs and of double categories generally. Something like this 
has been studied in the foundational paper \cite{shulman2008} under the guise of 
\emph{coextension}. It was mentioned there, but not studied in-depth due to 
being not as well behaved as the other two operations. And indeed this is not 
what we envision in this paper. For coextensions are had by asking for the 
equipment $\dbl{D}_1\to\dbl{D}_0\times\dbl{D}_0$ to be also a \emph{cofibration} 
which is too strong for what we have in mind. Rather, we take the approach of 
\emph{fibrational semantics} \cite{jacobs1999} and will ask merely that our 
equipments $\dbl{D}_1\to\dbl{D}_0\times\dbl{D}_0$ have right adjoints to 
restriction functors and that these satisfy a Beck-Chevalley condition. This is 
exactly to require \emph{fibered products}. We shall see in more detail how this 
recovers exactly the queries of interest below in 
\cref{section:universal-queries}. For this we first review quantification in 
more detail.

\section{Quantification and Modality}
\label{section:varieties-of-quantification}

It is well-known that quantification provides left and right adjoints to 
substitution \cite{lawvere2006,maclane1992}. For an ordinary set function 
$f\colon X\to Y$, there is the usual adjoint situation
  \begin{equation*}
    \begin{tikzcd}
      PX && PY
      \arrow["{\forall_f}", shift left=3, from=1-1, to=1-3]
      \arrow["{\exists_f}"', shift right=3, from=1-1, to=1-3]
      \arrow["{f^*}"{description}, from=1-3, to=1-1]
    \end{tikzcd}\qquad\qquad \exists_f\dashv f^*\dashv\forall_f
  \end{equation*}
where
  \begin{equation*}
    \exists_f(S) = \lbrace y\in Y\mid f(x) = y \text{ for some } x\in S\rbrace
  \end{equation*}
  \begin{equation*}
    \forall_f(S) = \lbrace y\in Y\mid \text{ for all } x\in X \text{ if } f(x) = y \text{ then } x\in S\rbrace.
  \end{equation*}
Ordinary quantification is the special case taking 
$f = \pi\colon X\times Y\to Y$, the usual projection to the second factor of the 
cartesian product. In this case, the quantified subsets are 
$R\subset X\times Y$, that is, binary relations. That $x$ and $y$ are 
$R$-related is written $R(x,y)$ or $x\leq_Ry$. Quantification is then:
  \begin{equation*}
    \exists_\pi(R) = \lbrace y\in Y\mid R(x,y) \text{ for some } x\in X\rbrace
  \end{equation*}
  \begin{equation*}
    \forall_\pi(R) = \lbrace y\in Y\mid R(x,y) \text{ for all } x\in X\rbrace.
  \end{equation*}
Typically, the left variable is bound. Note also in particular that if $Y=1$, we 
have closed formulas, that is, truth values.  

Now, a relation $R\colon X\proto Y$ might also be thought of as set of pairs of 
entities and observations or attribute values, or when $X= Y$ as a 
non-deterministic system and $x\leq_Rx'$ as representing a causal or temporal 
transition or statement of accessibility. Now, the set of all $x\in X$ for which 
some subsequent state or attribute value satisfies a proposition $S\subset Y$ is 
written 
  \begin{equation*} \label{eqn:diamond-definition}
    \diamondsuit S = \lbrace x\mid S(y) \text{ for some } x\leq_R y \rbrace.
  \end{equation*}
This is exactly the modal possibility operator familiar from the topological 
semantics given by the Alexandrov Topology. Likewise there is the set of those 
elements of $X$, all of whose subsequent states or attribute values satisfy 
$S\subset Y$, written 
  \begin{equation*}
    \Box S = \lbrace x\mid S(y) \text{ for all } x\leq_R y\rbrace.
  \end{equation*}
This is modal necessity in the same topological semantics. Thus, in the former 
case, an element of $\diamondsuit S$ is an element of $X$ for which $S$ is 
possible, in the sense that it is true of some later accessible state. In the 
latter case, an element of $\Box S$ is an element of $X$ for which $S$ is 
necessarily true, in the sense that any subsequent accessible state satisfies 
$S$. Note that this data of a relation $R$ and proposition $S\subset Y$ is 
exactly the set up of our olog:
  
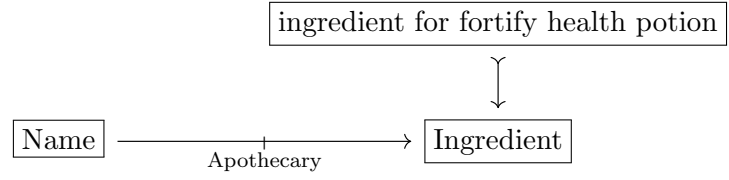
\begin{equation*}
      \begin{tikzcd}
        && {\fbox{\text{ingredient for fortify health potion}}} \\
        {\fbox{Name}} && {\fbox{Ingredient}}
        \arrow[tail, from=1-3, to=2-3]
        \arrow["{\text{Apothecary}}"'{inner sep=.8ex}, "\shortmid"{marking}, from=2-1, to=2-3]
      \end{tikzcd}.
    \end{equation*}
Recall too the table instancing the bottom relation:
  \begin{equation*}
    \begin{tabular}{| l | l | }
        \hline\multicolumn{2}{| c |}{\bf Apothecary}\\
        \hline {\bf Name }&{\bf Ingredient}\\
        \hline Arcadia's Cauldron & giant's toe \\
        \hline Arcadia's Cauldron & wheat \\
        \hline Arcadia's Cauldron & void salts \\
        \hline Elgrim's Elixers & giant's toe \\
        \hline Elgrim's Elixers & wheat \\
        \hline The White Phial & wheat \\
        \hline
    \end{tabular}.
  \end{equation*}
Our list $S$ again has just wheat and giant's toe. Directly from the operations 
above we have six queries summarized in \cref{fig:quant-queries}. In the figure, 
$R$ denotes the relation and $R^\dagger$ is the reverse relation. The variables 
$x$ and $y$ range over names and over ingredients, respectively. 
\begin{figure}[H]
\begin{center}
\renewcommand{\arraystretch}{1.5}

\newlength{\boxheight}
\setlength{\boxheight}{3cm}

\tcbset{
  querybox/.style={
    boxrule=0.5mm, 
    sharp corners, 
    height=\boxheight,
    colback=white,
    colframe=black,
    coltitle=white,
    colbacktitle=black,
    halign=center, 
    valign=center  
  }
}

\begin{tabularx}{\linewidth}{X X}

\begin{tcolorbox}[querybox, title=Result of Existential Query]
  $\begin{aligned}
    \exists_\pi R &= \lbrace y\mid R(x,y) \text{ for some } x\in X\rbrace \\
                  &= \lbrace \text{ingredients someone has}\rbrace \\
                  &= \lbrace\text{wheat}, \text{giant's toe},\text{void salts}\rbrace
  \end{aligned}$
\end{tcolorbox}
&
\begin{tcolorbox}[querybox, title=Result of Universal Query]
  $\begin{aligned}
    \forall_\pi R &= \lbrace y\mid R(x,y) \text{ for all } x\in X\rbrace \\
                  &= \lbrace \text{ingredients everyone has}\rbrace \\
                  &= \lbrace\text{wheat}\rbrace
  \end{aligned}$
\end{tcolorbox}
\\[2ex]

\begin{tcolorbox}[querybox, title=Result of Reverse Existential Query]
  $\begin{aligned}
    \exists_\pi R^\dagger &= \lbrace x \mid R(x,y) \text{ for some } y\in Y\rbrace \\
                  &= \lbrace \text{those with an $R$-ingredient}\rbrace \\
                  &= \lbrace\text{Arcadia's}, \text{Elgrim's}, \text{Phial}\rbrace
  \end{aligned}$
\end{tcolorbox}
&
\begin{tcolorbox}[querybox, title=Result of Reverse Universal Query]
  $\begin{aligned}
    \forall_\pi R^\dagger &= \lbrace x\mid R(x,y) \text{ for all } y\in Y\rbrace \\
                  &= \lbrace \text{those with all $R$-ingredients}\rbrace \\
                  &= \lbrace\text{Arcadia's}\rbrace
  \end{aligned}$
\end{tcolorbox}
\\[2ex]

\begin{tcolorbox}[querybox, title=Result of Possibility Query (Liveness)]
  $\begin{aligned}
  \diamondsuit S &=\lbrace x\mid S(y) \text{ for some } x\leq_R y \rbrace \\
  &= \lbrace x \mid x \text{ has an ingredient on } S \rbrace \\ 
  &= \lbrace \text{Arcadia's}, \text{Elgrim's}, \text{Phial}\rbrace
  \end{aligned}$
\end{tcolorbox}
&
\begin{tcolorbox}[querybox, title=Results of Necessity Query (Safety)]
  $\begin{aligned}
  \Box S &= \lbrace x\mid S(y) \text{ for all } x\leq_R y\rbrace\\ 
  &= \lbrace x \mid x \text{ has only ingredients on } S \rbrace \\ 
  &= \lbrace \text{Elgrim's}, \text{Phial}\rbrace
  \end{aligned}$
\end{tcolorbox}

\end{tabularx}
\end{center}
\caption{Results of six quantification queries.}
\label{fig:quant-queries}
\end{figure}
Notice some curiosities. The first is that $R$ is given. Thus, the existential 
queries are just the select column queries returning either all the merchants or 
all the ingredients specified by $R$. These are already studied in 
\cite{lambert2025}. The universal queries are new and more interesting, but 
neither returns our list query since $S$ does not even enter into the 
construction. The modal queries as well are new. Although they seemingly 
superficial resembling the division query, they are actually quite different. 
The possibility query returns the same as the existential query only as an 
accident of $R$. If The White Phial was paired with something other than wheat, 
for example, the results would be different. Note as well that the necessity 
query is not the list query we are looking for. It returns the set of those 
merchants having only the ingredients on our list. This is a query with utility, 
but as far as our shopping list goes it does not matter whether the returned 
merchant has other items. In fact, suppose that one closer has our ingredients 
but also a lot of others. We would naturally prefer to go there. The possibility 
construction is very close to presenting a \emph{liveness property} whereas the 
necessity construction is essentially phrasing a \emph{safety property} 
\cite{lamport1977}, \cite{alpern1985}, \cite{alpern1987}. The former means 
\emph{something good eventually happens} (that is, eventually values are on the 
list $S$); and the latter means \emph{nothing bad happens} (i.e. values never 
leave the list $S$). Combined with negation 
(\cref{section:cocartesian-negation}) we have something that looks more 
explicitly like safety: an execution never enters the defined \emph{bad region} 
in that it always remains in its complement.

It will turn out that the list query is given by a pair of operations, one of 
which is universal quantification, applied in sequence. And so, to enrich the 
data schema to handle this, the goal now is to formulate all of the above 
operations purely double-categorically and axiomatize them in our generated 
ologs. Ultimately, we will allow the data-instancing and double-categorical 
functorial semantics to perform the desired queries.

\section{Existential Queries}
\label{section:existential-queries}

Before casting universal queries in double-categorical formalism, we look more 
closely at existential-type queries. The main consideration is the possibility 
query in \cref{fig:quant-queries}. In this section we work in the setting of a 
generic cartesian equipment $\dbl D$ and gradually add further assumptions as 
needed. The discussion of existential operations in terms of \emph{Kan 
quantification} (\cref{discussion:kan-quantification}) motivates the eventual 
adoption of right adjoints to substitution as that machinery capturing division.

\begin{con}
  The set-up we are entertaining is that of having a subobject or subtype 
  $\phi\colon s\rightarrowtail y$ thought of as specifying in the associated 
  data some subcollection or sublist of items in the ambient list instancing 
  $y$. We have already seen that in $\Rel$, merely composing has the effect of 
  an existential query. This can be formalized $\dbl{D}$ by first viewing $s$ as 
  an extension and then composition with $r$ on the left as in 
    \begin{equation*}
      \begin{tikzcd}
        & s & s \\
        x & y & 1
        \arrow[""{name=0, anchor=center, inner sep=0}, "{\id_S}"{inner sep=.8ex}, "\shortmid"{marking}, from=1-2, to=1-3]
        \arrow["\phi"', tail, from=1-2, to=2-2]
        \arrow["{!}", from=1-3, to=2-3]
        \arrow["r"'{inner sep=.8ex}, "\shortmid"{marking}, from=2-1, to=2-2]
        \arrow[""{name=1, anchor=center, inner sep=0}, "{\hat\phi}"'{inner sep=.8ex}, "\shortmid"{marking}, from=2-2, to=2-3]
        \arrow["{\text{ext}}"{description}, draw=none, from=0, to=1]
      \end{tikzcd}.
    \end{equation*}
  Here we are using the notation $\hat\phi$ to stand for the codomain of the 
  extension above which is a composite of the conjoint of $\phi$ with the 
  companion of the unique $s\to 1$. This is essentially converting a subtype or 
  term to a one-sided proposition in the logic and semantically it is reversing 
  the operation of taking a tabulator. In $\Rel$, the composite of proarrows on 
  the bottom is precisely modal possibility, $\diamondsuit_r\phi$. An equivalent 
  $\Rel$-construction is to form a pullback along $\phi$ and then take an image 
  along $s\to 1$. This is formalized in any cartesian equipment as 
    \begin{equation*}
      \begin{tikzcd}
        x & s \\
        x & y
        \arrow[""{name=0, anchor=center, inner sep=0}, "{r\odot \phi^*}"{inner sep=.8ex}, "\shortmid"{marking}, from=1-1, to=1-2]
        \arrow[equals, from=1-1, to=2-1]
        \arrow["\phi", tail, from=1-2, to=2-2]
        \arrow[""{name=1, anchor=center, inner sep=0}, "r"'{inner sep=.8ex}, "\shortmid"{marking}, from=2-1, to=2-2]
        \arrow["{\text{res}}"{description}, draw=none, from=0, to=1]
      \end{tikzcd} \qquad\leadsto\qquad 
      \begin{tikzcd}
        x & s \\
        x & 1
        \arrow[""{name=0, anchor=center, inner sep=0}, "{r\odot \phi^*}"{inner sep=.8ex}, "\shortmid"{marking}, from=1-1, to=1-2]
        \arrow[equals, from=1-1, to=2-1]
        \arrow["{!}", from=1-2, to=2-2]
        \arrow[""{name=1, anchor=center, inner sep=0}, "{\diamondsuit_r \phi}"'{inner sep=.8ex}, "\shortmid"{marking}, from=2-1, to=2-2]
        \arrow["{\text{ext}}"{description}, draw=none, from=0, to=1]
      \end{tikzcd}.
    \end{equation*}
  In $\Rel$, the top of the restriction on the right is exactly the set of pairs 
  $(x,y)$ where $x\leq_Ry$ under $R$. The image then binds over $y$ via 
  existential quantification. Now, in $\dbl{D}$ there is a globular comparison 
  cell $\diamondsuit_r \phi \Rightarrow r\odot \hat\phi$ since 
  $\diamondsuit_r \phi$ is constructed as an extension: 
    \begin{equation} \label{equation:globular-cell}
      \begin{tikzcd}
        x && s \\
        x & s & s \\
        x & y & 1
        \arrow[""{name=0, anchor=center, inner sep=0}, "{r\odot \phi^*}"{inner sep=.8ex}, "\shortmid"{marking}, from=1-1, to=1-3]
        \arrow[equals, from=1-1, to=2-1]
        \arrow[from=1-3, to=2-3]
        \arrow[""{name=1, anchor=center, inner sep=0}, "{r\odot \phi^*}"{inner sep=.8ex}, "\shortmid"{marking}, from=2-1, to=2-2]
        \arrow["1"', equals, from=2-1, to=3-1]
        \arrow[""{name=2, anchor=center, inner sep=0}, "\id"{inner sep=.8ex}, "\shortmid"{marking}, from=2-2, to=2-3]
        \arrow["\phi", from=2-2, to=3-2]
        \arrow["{!}", from=2-3, to=3-3]
        \arrow[""{name=3, anchor=center, inner sep=0}, "r"'{inner sep=.8ex}, "\shortmid"{marking}, from=3-1, to=3-2]
        \arrow[""{name=4, anchor=center, inner sep=0}, "{\hat\phi}"'{inner sep=.8ex}, "\shortmid"{marking}, from=3-2, to=3-3]
        \arrow["\cong"{description}, draw=none, from=0, to=2-2]
        \arrow["{\text{res}}"{description}, draw=none, from=1, to=3]
        \arrow["{\text{ext}}"{description}, draw=none, from=2, to=4]
      \end{tikzcd} \qquad = \qquad 
      \begin{tikzcd}
        x & s \\
        x & 1 \\
        x & 1
        \arrow[""{name=0, anchor=center, inner sep=0}, "{r\odot \phi^*}"{inner sep=.8ex}, "\shortmid"{marking}, from=1-1, to=1-2]
        \arrow[equals, from=1-1, to=2-1]
        \arrow["{!}", from=1-2, to=2-2]
        \arrow[""{name=1, anchor=center, inner sep=0}, "{\diamondsuit_r \phi}"'{inner sep=.8ex}, "\shortmid"{marking}, from=2-1, to=2-2]
        \arrow[equals, from=2-1, to=3-1]
        \arrow[equals, from=2-2, to=3-2]
        \arrow[""{name=2, anchor=center, inner sep=0}, "{r\odot \hat\phi}"'{inner sep=.8ex}, "\shortmid"{marking}, from=3-1, to=3-2]
        \arrow["{\text{ext}}"{description}, draw=none, from=0, to=1]
        \arrow["{\exists\,!}"{description, pos=0.6}, draw=none, from=1, to=2]
      \end{tikzcd}.
    \end{equation}
  Now, it can be seen directly that this is an invertible cell when $\dbl{D}$ is 
  $\Mat(\cat{V})$ for suitable monoidal $\cat V$. This covers relations and 
  spans. Now, in fact they agree in any equipment.
\end{con}

\begin{prop} \label{prop:possibility-rewrite-rule-1}
  The globular cell $\diamondsuit_r \phi \Rightarrow r\odot \hat\phi$ induced in \cref{equation:globular-cell} is an isomorphism in any equipment. Accordingly, the two potential definitions of possibility agree.
\end{prop}
\begin{proof}
  It suffices to show that the two sides of \cref{equation:globular-cell} 
  compute the same extension. But this is an exercise is rearranging the 
  left-hand side via the properties of restrictions and extensions and their 
  construction via companions and conjoints \cite[\S 4]{shulman2008}.
\end{proof}

So, although the constructions are thus essentially the same, we make a choice 
as to which is basic. We insist on starting with the subtype $\phi\colon s\to y$ 
consistent with the set-up of \cref{section-ologs}. And so in either case there 
are two operations to perform the query. There is a slight semantic advantage to 
the latter. For in the case of $\Span$, although the appropriate element 
$x\in X$ is implicit in the composite as $d(r) = x$, our preference is to keep 
the data as a triple $((x,a),r) = (x,a,r)$ as in the second way of computing the 
possibility operator. This way all the information is recorded in the apex of 
the span and the two legs could potentially be forgotten.

\begin{defi}[Possibility] \label{def:possibility-operator}
  Let $\phi\colon s\rightarrowtail y$ denote any monic arrow in the equipment $\dbl{D}$. The proposition \textbf{possibly} $\phi$ is then formed as an restriction followed by an extension:
    \begin{equation*}
      \begin{tikzcd}
        x & s \\
        x & y
        \arrow[""{name=0, anchor=center, inner sep=0}, "{r\odot \phi^*}"{inner sep=.8ex}, "\shortmid"{marking}, from=1-1, to=1-2]
        \arrow[equals, from=1-1, to=2-1]
        \arrow["\phi", tail, from=1-2, to=2-2]
        \arrow[""{name=1, anchor=center, inner sep=0}, "r"'{inner sep=.8ex}, "\shortmid"{marking}, from=2-1, to=2-2]
        \arrow["{\text{res}}"{description}, draw=none, from=0, to=1]
      \end{tikzcd} \qquad\leadsto\qquad 
      \begin{tikzcd}
        x & s \\
        x & 1
        \arrow[""{name=0, anchor=center, inner sep=0}, "{r\odot \phi^*}"{inner sep=.8ex}, "\shortmid"{marking}, from=1-1, to=1-2]
        \arrow[equals, from=1-1, to=2-1]
        \arrow["{!}", from=1-2, to=2-2]
        \arrow[""{name=1, anchor=center, inner sep=0}, "{\diamondsuit_r \phi}"'{inner sep=.8ex}, "\shortmid"{marking}, from=2-1, to=2-2]
        \arrow["{\text{ext}}"{description}, draw=none, from=0, to=1]
      \end{tikzcd}.
    \end{equation*}
  If $r$ is understood, write just `$\diamondsuit\phi$' for this proposition.
\end{defi}

Note that then in $\Rel$-semantics, this operation actually recovers all the 
existential queries in \cref{fig:quant-queries} by combining with restrictions 
or reverse operators; that is, each is either an image or a composition. But the 
content of \cref{prop:possibility-rewrite-rule-1} is that images and 
compositions have essentially the same effect in suitable cases in general 
equipments. We close the section now with a high-level description of the common 
setting embracing all of these cases. This leads to our abstract formulation of 
universal queries in \cref{section:universal-queries} below.

\begin{dis}[Quantification Motivation] \label{discussion:kan-quantification}
  All the instances of existential quantification discussed above arise from the 
  \emph{fibrational semantics} of the equipment $\dbl{D}$. We have the general 
  picture:
    \begin{equation} \label{eqn:left-adjoint-to-restriction}
      \coprod_{f,g}\colon\dbl{D}(a,b)\rightleftarrows\dbl{D}(x,y)\colon (f,g)^*\qquad\qquad\coprod_{f,g}\dashv (f,g)^*
    \end{equation}
  for any arrows $f\colon a\to x$ and $g\colon b\to y$ of $\dbl{D}$. The hom 
  categories are the fibers of source-target fibration 
  $\langle\src,\tgt\rangle\colon\dbl{D}_1\to\dbl{D}_0\times\dbl{D}_0$ over the 
  pairs of objects $(a,b)$ and $(x,y)$. And the coproduct denotes the left 
  adjoint to restriction along $f$ and $g$. All instances of existential 
  quantification in \cref{fig:quant-queries} are recoverable in this framework. 
  On the one hand, for $\phi\colon s\to y$ we have the functors
    \begin{equation} \label{eqn:possibility-via-doctrines}
      \begin{tikzcd}
        {\dbl{D}(x,y)} & {\dbl{D}(x,s)} & {\dbl{D}(x,1)}
        \arrow["{(x,\phi)^*}", from=1-1, to=1-2]
        \arrow["{\coprod_{x,!}}", from=1-2, to=1-3]
      \end{tikzcd}\qquad\qquad \coprod_{x,!}(x,s)^*(r) = \diamondsuit_r \phi
    \end{equation}
  as in \cref{def:possibility-operator}. On the other hand, ordinary existential 
  quantification is just an application of a left adjoint without the subtyping 
  and restriction:
    \begin{equation*}
      \begin{tikzcd}
        {\dbl{D}(x,y)} & {\dbl{D}(x,1)}
        \arrow["{\coprod_{x,!}}", from=1-1, to=1-2]
      \end{tikzcd} \qquad\qquad \coprod_{x,!}(r) =: \exists y . r
    \end{equation*}
  as has already been observed in the case of $\Rel$. A reverse operator thus 
  covers the last existential query in \cref{fig:quant-queries}. Now, in each 
  case, the left adjoints used are special cases of \emph{Kan quantification} 
  \cite[\S 3]{lawvere2002}. These are also fragments of the bicategorical closed 
  structure discussed in \cite[\S 5]{shulman2008}. That is, these are species of 
  one-sided adjunctions 
    \begin{equation} \label{eqn:one-sided-left-adjoint-to-restriction}
      \coprod_{x,h}\colon\dbl{D}(x,z)\rightleftarrows\dbl{D}(x,w)\colon (x,h)^*\qquad\coprod_{x,h}\dashv (x,h)^*
    \end{equation}
  for various choices of $h\colon z\to w$. Owning to the nature of the 
  computation of restrictions and extensions via companions and conjoints 
  \cite[Theorem 4.1]{shulman2008}, the two functors in the adjunction 
  immediately above are simply: \emph{postcompose with a companion or conjoint} 
  as in 
    \begin{equation} \label{eqn:one-sided-left-adjoint-to-restriction2}
      -\odot h_!\colon\dbl{D}(x,z)\rightleftarrows\dbl{D}(x,w)\colon -\odot h^* \qquad -\odot h_! \dashv -\odot h^*.
    \end{equation}
  Retaining the coproduct notation for the left adjoint, the adjunction thus 
  expresses an existential quantification introduction rule. That is, it 
  presents a bijection
    \begin{equation*}
      \begin{tikzcd}
        x & z \\
        x & w
        \arrow[""{name=0, anchor=center, inner sep=0}, "{\coprod_{x,h}(m)}"{inner sep=.8ex}, "\shortmid"{marking}, from=1-1, to=1-2]
        \arrow[equals, from=1-1, to=2-1]
        \arrow["h", from=1-2, to=2-2]
        \arrow[""{name=1, anchor=center, inner sep=0}, "n"'{inner sep=.8ex}, "\shortmid"{marking}, from=2-1, to=2-2]
        \arrow[between={0.3}{0.7}, Rightarrow, from=0, to=1]
      \end{tikzcd} \quad\leftrightarrow\quad 
      \begin{tikzcd}
        x && z \\
        x & w & z
        \arrow[""{name=0, anchor=center, inner sep=0}, "m"{inner sep=.8ex}, "\shortmid"{marking}, from=1-1, to=1-3]
        \arrow[equals, from=1-1, to=2-1]
        \arrow[equals, from=1-3, to=2-3]
        \arrow["n"'{inner sep=.8ex}, "\shortmid"{marking}, from=2-1, to=2-2]
        \arrow["{h^*}"'{inner sep=.8ex}, "\shortmid"{marking}, from=2-2, to=2-3]
        \arrow[between={0.4}{0.8}, Rightarrow, from=0, to=2-2]
      \end{tikzcd}.
    \end{equation*}
  Specializing to the locally posetal case and where $h$ is the unique $y\to 1$ 
  as above and writing the traditional `$\exists$' for the left adjoint, the 
  bijection expressing the universal property is thus precisely the usual 
  existential two-way rule:
    \begin{equation*}
      \begin{tikzcd}
        x & y \\
        x & 1
        \arrow[""{name=0, anchor=center, inner sep=0}, "{\exists_y(m)}"{inner sep=.8ex}, "\shortmid"{marking}, from=1-1, to=1-2]
        \arrow[equals, from=1-1, to=2-1]
        \arrow["{!}", from=1-2, to=2-2]
        \arrow[""{name=1, anchor=center, inner sep=0}, "n"'{inner sep=.8ex}, "\shortmid"{marking}, from=2-1, to=2-2]
        \arrow["\leq"{description}, draw=none, from=0, to=1]
      \end{tikzcd} \quad\leftrightarrow\quad
      \begin{tikzcd}
        x && y \\
        x & 1 & y
        \arrow[""{name=0, anchor=center, inner sep=0}, "m"{inner sep=.8ex}, "\shortmid"{marking}, from=1-1, to=1-3]
        \arrow[equals, from=1-1, to=2-1]
        \arrow[equals, from=1-3, to=2-3]
        \arrow["n"'{inner sep=.8ex}, "\shortmid"{marking}, from=2-1, to=2-2]
        \arrow["{y^*}"'{inner sep=.8ex}, "\shortmid"{marking}, from=2-2, to=2-3]
        \arrow["\leq"{description, pos=0.6}, draw=none, from=0, to=2-2]
      \end{tikzcd}\qquad \qquad
      \inferrule{\exists_y(m)\leq n}{m\leq n(y^*)}
    \end{equation*}
  where we have written $y^*$ for the conjoint of the unique $y\to 1$. This 
  existential two-way rule motivates the introduction of dependent products as 
  right adjoints to substitution in order to capture the implicit universal 
  quantification in relational division queries.
\end{dis}

\section{Universal Queries}
\label{section:universal-queries}

Now we turn to universal quantification starting from fibrational semantics and 
Kan quantification as framed in \cref{discussion:kan-quantification}. That the 
restriction functor $(f,g)^*\colon \dbl{D}(x,y)\to\dbl{D}(a,b)$ has a right 
adjoint is the statement that there exists a functor 
  \begin{equation*}
    \prod_{f,g}\colon \dbl{D}(a,b)\to\dbl{D}(x,y)
  \end{equation*}
and natural isomorphisms
  \begin{equation*}
    \dbl{D}(a,b)(f_!mg^*,s)\cong \dbl{D}(x,y)(m,\prod_{f,g}(s))
  \end{equation*}
establishing natural families of bijections of globular cells 
    \begin{equation*}
      \inferrule{f_!mg^*\Rightarrow s}{m\Rightarrow \prod_{f,g}(s)}
    \end{equation*}
given in each direction by application of the appropriate functor and then 
composition with a unit or counit as the case may be. Notice again that this 
specializes to a bidirectional universal quantification rule. The elementary 
characterization of this right adjoint as a universal in this context is the 
following.
\begin{prop}
  The restriction functor $(f,g)^*\colon \dbl{D}(x,y)\to\dbl{D}(a,b)$ has a right adjoint if, and only if, for each $s\colon a\proto b$ there is a specified proarrow $\forall_{f,g}(s)$ together with a globular cell $\epsilon_s$ of the form 
    \begin{equation*}
      \begin{tikzcd}
        a && b \\
        x && y
        \arrow[""{name=0, anchor=center, inner sep=0}, "{f_!(\forall_{f,g}(s))g^*}"{inner sep=.8ex}, "\shortmid"{marking}, from=1-1, to=1-3]
        \arrow["f"', from=1-1, to=2-1]
        \arrow["g", from=1-3, to=2-3]
        \arrow[""{name=1, anchor=center, inner sep=0}, "{\forall_{f,g}(s)}"'{inner sep=.8ex}, "\shortmid"{marking}, from=2-1, to=2-3]
        \arrow["{\text{res}}"{description}, draw=none, from=0, to=1]
      \end{tikzcd} \qquad \leadsto \qquad
      \begin{tikzcd}
        a && b \\
        a && b
        \arrow[""{name=0, anchor=center, inner sep=0}, "{f_!(\forall_{f,g}(s))g^*}"{inner sep=.8ex}, "\shortmid"{marking}, from=1-1, to=1-3]
        \arrow[equals, from=1-1, to=2-1]
        \arrow[equals, from=1-3, to=2-3]
        \arrow[""{name=1, anchor=center, inner sep=0}, "s"'{inner sep=.8ex}, "\shortmid"{marking}, from=2-1, to=2-3]
        \arrow["{\epsilon_s}"{description}, draw=none, from=0, to=1]
      \end{tikzcd}
    \end{equation*}
  such that for any other globular cell from a restriction along $f$ and $g$, say, of the form 
    \begin{equation*}
      \begin{tikzcd}
        a && b \\
        x && y
        \arrow[""{name=0, anchor=center, inner sep=0}, "{f_!wg^*}"{inner sep=.8ex}, "\shortmid"{marking}, from=1-1, to=1-3]
        \arrow["f"', from=1-1, to=2-1]
        \arrow["g", from=1-3, to=2-3]
        \arrow[""{name=1, anchor=center, inner sep=0}, "w"'{inner sep=.8ex}, "\shortmid"{marking}, from=2-1, to=2-3]
        \arrow["{\text{res}}"{description}, draw=none, from=0, to=1]
      \end{tikzcd} \qquad\leadsto\qquad
      \begin{tikzcd}
        a && b \\
        a && b
        \arrow[""{name=0, anchor=center, inner sep=0}, "{f_!wg^*}"{inner sep=.8ex}, "\shortmid"{marking}, from=1-1, to=1-3]
        \arrow[equals, from=1-1, to=2-1]
        \arrow[equals, from=1-3, to=2-3]
        \arrow[""{name=1, anchor=center, inner sep=0}, "s"'{inner sep=.8ex}, "\shortmid"{marking}, from=2-1, to=2-3]
        \arrow["\xi"{description}, draw=none, from=0, to=1]
      \end{tikzcd}
    \end{equation*}
  there is a unique cell $\hat \xi$ for which $h$ factors through $\epsilon_s$ as in 
    \begin{equation*}
      \begin{tikzcd}
        a && b \\
        a && b \\
        a && b
        \arrow[""{name=0, anchor=center, inner sep=0}, "{f_!wg^*}"{inner sep=.8ex}, "\shortmid"{marking}, from=1-1, to=1-3]
        \arrow[equals, from=1-1, to=2-1]
        \arrow[equals, from=1-3, to=2-3]
        \arrow[""{name=1, anchor=center, inner sep=0}, "{f_!(\forall_{f,g}(s))g^*}"{inner sep=.8ex}, "\shortmid"{marking}, from=2-1, to=2-3]
        \arrow[equals, from=2-1, to=3-1]
        \arrow[equals, from=2-3, to=3-3]
        \arrow[""{name=2, anchor=center, inner sep=0}, "s"'{inner sep=.8ex}, "\shortmid"{marking}, from=3-1, to=3-3]
        \arrow["{\hat\xi}"{description, pos=0.4}, draw=none, from=0, to=1]
        \arrow["{\epsilon_s}"{description}, draw=none, from=1, to=2]
      \end{tikzcd}  
      \qquad = \qquad
      \begin{tikzcd}
        a & b \\
        \\
        a & b
        \arrow[""{name=0, anchor=center, inner sep=0}, "{f_!wg^*}"{inner sep=.8ex}, "\shortmid"{marking}, from=1-1, to=1-2]
        \arrow[equals, from=1-1, to=3-1]
        \arrow[equals, from=1-2, to=3-2]
        \arrow[""{name=1, anchor=center, inner sep=0}, "s"'{inner sep=.8ex}, "\shortmid"{marking}, from=3-1, to=3-2]
        \arrow["\xi"{description}, draw=none, from=0, to=1]
      \end{tikzcd}
    \end{equation*}
\end{prop}
\begin{proof}
  This is just the characterization of an adjunction in terms of the counit as a 
  universal \cite[\S IV.1, Theorem 2(iv)]{maclane1998}. Note that naturality of 
  $\epsilon_s$ in $s$ follows automatically.
\end{proof}

\begin{rem}
  There is a crucial conceptual and technical difference here in the fibrational 
  semantics \emph{vis-\'a-vis} the prior case of existential quantification. 
  This is that universal quantification is not neatly expressible via loose 
  composition with a companion or conjoint; nor is there some obvious third type 
  of proarrow $f_*$ associated to a morphism $f$ with its own accompanying 
  base-change cells and universal properties that would do this. For although 
  equipments $\dbl{D}$ involve both a fibration and opfibration 
  $\langle \src,\tgt\rangle\colon \dbl{D}_1\to\dbl{D}_0\times\dbl{D}_0$, this 
  same functor is not usually also a cofibration \cite{shulman2008}. This would 
  the structure required to provide some kind of \emph{corestriction} as a right 
  adjoint to restriction. Universal quantification thus must appear under some 
  different guise. Our approach of course has been that of fibrational 
  semantics: adjoints to substitution are the basic notion; that in one case 
  they happen to be computed by companions or conjoints is a nice feature of 
  equipments, but evidently does not carry over the others. Now, our axiomatized 
  notion is the following, based on \cite[\S1.9]{jacobs1999}. This is a 
  geometric/fibrational analogue of a \emph{quantification hyperdoctrine} 
  \cite{lawvere2006}, \cite{lawvere1970} instantiated here for double 
  categories. As defined in \cite{lambert2025}, a double olog is a small and 
  finitely-presented `double category of relations' \cite{lambert2022}. This is 
  in particular a locally posetal cartesian equipment.
\end{rem}

\begin{defi} \label{def:prod-double-olog}
  A $\prod$-\textbf{double olog} is a double olog where each substitution 
  functor has a right adjoint satisfying Beck-Chevalley.
\end{defi}

\begin{rem}
  Of course Beck-Chevalley holds automatically if the mere existence of 
  dependent products is assumed \emph{and} dependent coproducts are assumed to 
  satisfy Beck-Chevalley. This is by the usual mate calculus 
  \cite[\S 1.9]{jacobs1999}. In particular, a `double category of relations' 
  with strong tabulators and dependent products is a $\prod$-double olog 
  (\cref{lemma:beck-chevalley}). Certainly $\Rel$ has such adjoints since they 
  are inherited from $\Set$ and computed in the manner described in the 
  introduction. Beck-Chevalley follows or can be verified directly. Thus, given 
  an appropriate additional \emph{reverse} operator on proarrows, we have an 
  axiomatization of the two universal queries of \cref{fig:quant-queries} given 
  by application of the appropriate right adjoint.
\end{rem}

Our main application of universal quantification is, however, the list-query 
described \cref{section:relational-division} where we look for those proprietors 
having all the items on our shopping list. This can be done by a composition of 
a restriction followed by a universal quantifier right adjoint. Start with a 
given proarrow $r\colon x\proto y$ and a subtype $\phi\colon s\to y$ 
representing our list. Consider then the composite functor 
  \begin{equation} \label{eqn:list-via-doctrines}
    \begin{tikzcd}
      {\dbl{D}(x,y)} & {\dbl{D}(x,s)} & {\dbl{D}(x,1)}
      \arrow["{(x,\phi)^*}", from=1-1, to=1-2]
      \arrow["{\prod_{x,!}}", from=1-2, to=1-3]
    \end{tikzcd}
  \end{equation}
assuming that such right adjoints exist. Note that this is a direct analogue of 
the composite \cref{eqn:possibility-via-doctrines} which produced the modal 
query retrieving exactly those proprietors with some item on the list $S$. In 
the case of relations, this has the following effect.

\begin{exa} \label{example:recast-relational-division}
  Specializing \cref{eqn:list-via-doctrines} to $\Rel$, recall that the 
  restriction $R\odot \phi^*$ is the set of those pairs $(x,y)$ such that 
  $x\leq_R y$ and $y\in S$. Now, the map $X\times S\to X\times 1$ is up to 
  bijection between codomains the same as the first projection $X\times S\to X$. 
  So, using the computation of $\forall_{\pi_X}$ in 
  \cref{section:varieties-of-quantification}, we have 
    \begin{align*}
      \forall_{\pi_X}(R\odot \phi^*) &= \lbrace x\in X\mid \text{ for all } (x',y)\in X\times S \, (\pi(x',y) = x \text{ implies } (x',y)\in R\odot \phi^*)\rbrace \\
      &= \lbrace x\in X\mid \text{ for all } (x,y)\in X\times S \, (x,y)\in R\odot \phi^*\rbrace \\ 
      &= \lbrace x\in X\mid \text{ for all } y \in S \, (x,y)\in R\odot \phi^*\rbrace \\
      &= \lbrace x\in X\mid \text{ for all } y \in S \, (x\leq_R y)\rbrace
    \end{align*}
  That is, the result of the application of the composite of functors to $R$ is 
  the set of those $X$-elements $x$ for which everything in $S$ is $R$-related 
  to $x$. This is precisely to say that if $S$ is a list of ingredients and $X$ 
  is our list of proprietors, then the query returns the list of those who have 
  everything on the list (but could have other items).
\end{exa}

\begin{defi} \label{def:data-instance}
  A \emph{data instance} for a $\prod$-double olog $\dbl D$ is a double functor 
  $I\colon \dbl D\to\Rel$ that is a morphism of fibrations. In particular, $I$ 
  preserves substitution and its adjoints.
\end{defi}

\begin{rem}
  Although this construction thus achieves the desired effect and in this sense 
  the objective of the paper too, it leaves a gap in the narrative inasmuch as 
  it is not a modal query like the existential query relative to $S$ that it 
  mimics. The explanation is the relative flexibility of the existential 
  operator. That is, existential validity asks just that \emph{something} 
  satisfy the quantified formula, relation or proposition. This is a relatively 
  weak requirement. And so the result of the possibility query in 
  \cref{fig:quant-queries} is just the existential query $\exists$ applied to 
  $R\odot \phi^*$. Universal queries, by their nature surveying all entities 
  under the scope of the quantifier, are more stringent. Thus, 
  $\forall(R\odot \phi^*)$ and $\Box S$ produce dramatically different results: 
  the first case that of the shops with all ingredients on the list; and second 
  those with only ingredients on the list. That is, the first case quantifies 
  over list items and checks against inventory; the second quantifies over 
  inventory and checks against list items. The difficulty consists in the fact 
  that what needs to be returned in each case is the list of proprietors. The 
  built-in for $\forall$ as a right adjoint must match variances without the 
  mixing this would demand in one or the other case. That is, the comprehension 
  of $\forall$ must always match the codomain typing of operator as a right 
  adjoint. The object $R\odot \phi^*$ must be formed to compare shops with out 
  list items. But just, for example, switching the projection or reversing the 
  relation results in the codomain typing of $\forall$ to $S$-objects which will 
  return ingredients, not shops in the query. To recover this operator we turn 
  to tabulators in \cref{section:modality}.
\end{rem}

\section{Beck-Chevalley Condition}
\label{section:Beck-Chevalley}

As in \cite[\S 5.2]{aleiferi2018}, the (internal) Beck-Chevalley condition for 
an equipment $\dbl D$ says that given any pullback square on the left, the 
composite on the right is an isomorphism:
  \begin{equation*}
    \begin{tikzcd}
      x & y \\
      z & w
      \arrow["f", from=1-1, to=1-2]
      \arrow["h"', from=1-1, to=2-1]
      \arrow["\lrcorner"{anchor=center, pos=0.125}, draw=none, from=1-1, to=2-2]
      \arrow["k", from=1-2, to=2-2]
      \arrow["g"', from=2-1, to=2-2]
    \end{tikzcd}\qquad\qquad\qquad
    \begin{tikzcd}
      z & x & x & y \\
      z & z & y & y \\
      z & w & w & y
      \arrow[""{name=0, anchor=center, inner sep=0}, "{h^*}"{inner sep=.8ex}, "\shortmid"{marking}, from=1-1, to=1-2]
      \arrow[equals, from=1-1, to=2-1]
      \arrow["\id"{inner sep=.8ex}, "\shortmid"{marking}, from=1-2, to=1-3]
      \arrow["h", from=1-2, to=2-2]
      \arrow[""{name=1, anchor=center, inner sep=0}, "{f_!}"{inner sep=.8ex}, "\shortmid"{marking}, from=1-3, to=1-4]
      \arrow["f"', from=1-3, to=2-3]
      \arrow[equals, from=1-4, to=2-4]
      \arrow[""{name=2, anchor=center, inner sep=0}, "\shortmid"{marking}, from=2-1, to=2-2]
      \arrow[equals, from=2-1, to=3-1]
      \arrow["g", from=2-2, to=3-2]
      \arrow[""{name=3, anchor=center, inner sep=0}, "\shortmid"{marking}, from=2-3, to=2-4]
      \arrow["k"', from=2-3, to=3-3]
      \arrow[equals, from=2-4, to=3-4]
      \arrow[""{name=4, anchor=center, inner sep=0}, "{g_!}"'{inner sep=.8ex}, "\shortmid"{marking}, from=3-1, to=3-2]
      \arrow["\id"'{inner sep=.8ex}, "\shortmid"{marking}, from=3-2, to=3-3]
      \arrow[""{name=5, anchor=center, inner sep=0}, "{k^*}"'{inner sep=.8ex}, "\shortmid"{marking}, from=3-3, to=3-4]
      \arrow["\epsilon"{description}, draw=none, from=0, to=2]
      \arrow["\rho"{description}, draw=none, from=1, to=3]
      \arrow["\eta"{description}, draw=none, from=2, to=4]
      \arrow["\delta"{description}, draw=none, from=3, to=5]
    \end{tikzcd}.
  \end{equation*}
The cells labeled with Greek letters are the canonical base change cells coming 
with companions and conjoints. The main result of this section is that internal 
Beck-Chevalley implies the usual (external) Beck-Chevalley condition for the 
left adjoints to restriction in the context of fibrations 
\cite[\S 1.9]{jacobs1999}. As an upshot, this implies that any `double category 
of relations' with \emph{strong tabulators} (\cref{def:strong-tabulator}) 
satisfies external Beck-Chevalley automatically.

\begin{lem} \label{lemma:beck-chevalley}
  If $\dbl D$ is a cartesian equipment satisfying internal Beck-Chevalley, then 
  the source-target fibration 
  $\langle\src,\tgt\rangle\colon\dbl{D}_1\to\dbl{D}_0\times\dbl{D}_0$ satisfies 
  the usual Beck-Chevalley condition for the left adjoints to restriction along 
  pullback squares. In particular, any `double category of relations' with 
  strong tabulators satisfies the usual external Beck-Chevalley condition.
\end{lem}
\begin{proof}
  The statement of the result is that the canonical cell 
    \begin{equation*}
      \coprod_{f,g}(h,k)^*(m)\Rightarrow (u,v)^*\coprod_{p,q}(m)
    \end{equation*}
  arising from an appropriate pullback square in $\dbl D_0\times\dbl D_0$ is an 
  iso. But this is just a matter of the definition of the construction of 
  restrictions and extensions using companions and conjoints and the assumed 
  internal Beck-Chevalley condition summarized above. The last statement is a 
  consequence of \cite[Lemma 5.2.3]{aleiferi2018} which shows that any unit-pure 
  equipment with strong tabulators has pullback satisfying internal 
  Beck-Chevalley. The key assumption for this argument is \emph{unit-purity}. 
  But this is shown in \cite[Lemma 4.1.9]{hoshino2025} always to hold in any 
  double category where every object is discrete.
\end{proof}

\begin{rem}[Optimization Rewrite Rule] \label{remark:Beck-Chevalley-Optimization}
  External Beck-Chevalley as above is a rewrite rule for compound queries. This 
  is a \emph{filter-pushdown} optimization rule. The extension first on the 
  right above acts as an abstraction/collapse or extraction of metadata before 
  computing the filter/restriction. The rule says that this can more efficiently 
  be done by \emph{filtering first} to restrict the number of rows in the table 
  before the collapse. This is a standard and highly effective optimization in 
  query planning in relational databases which is thus available in any double 
  olog with strong tabulators by \cref{lemma:beck-chevalley}. Importantly, it is 
  a property of any such database schema that such optimization always 
  type-checks and is always error-free.
\end{rem}

\begin{exa} \label{example:beck-chevalley}
  This rule is illustrated by the following. Start with two tables: one 
  assigning to each individual his or her clan membership and another viewing 
  each clan with its factional loyalty:
    \begin{equation*}
      \begin{tabular}{| l | l | }
          \hline\multicolumn{2}{| c |}{\bf Membership}\\
          \hline {\bf Individual }&{\bf Clan}\\
          \hline  Asgeir & Snow-Shod  \\
          \hline  Eorlund & Gray-Mane  \\
          \hline  Galmar & Stone-Fist  \\
          \hline  Idgrod & Battle-Born  \\
          \hline  Nilsine & Shatter-Shield \\
          \hline  Torbjorn & Shatter-Shield  \\
          \hline  Vignar & Gray-Mane  \\
          \hline
      \end{tabular} \qquad\qquad
      \begin{tabular}{| l | l | }
          \hline\multicolumn{2}{| c |}{\bf Loyalty}\\
          \hline {\bf Clan }&{\bf Faction}\\
          \hline  Battle-Born & Imperial  \\
          \hline  Gray-Mane & Stormcloak  \\
          \hline  Shatter-Shield & Imperial  \\
          \hline  Snow-Shod & Stormcloak  \\
          \hline  Stone-Fist & Stormcloack \\
          \hline
      \end{tabular}.
    \end{equation*}
  Suppose that the Membership table is a relation extracted perhaps from some 
  larger dataset of various affiliations of relevant individuals; suppose that 
  the Loyalty table is functional metadata. There are then evidently two 
  schematized morphisms in the olog; one is a proarrow and one is an ordinary 
  arrow:
    \begin{equation*}
      \begin{tikzcd}
        {\fbox{Individual}} && {\fbox{Clan}}
        \arrow["{\text{Membership}}"{inner sep=.8ex}, "\shortmid"{marking}, from=1-1, to=1-3]
      \end{tikzcd}\qquad\qquad
      \begin{tikzcd}
        {\fbox{Clan}} && {\fbox{Faction}}
        \arrow["{\text{Loyalty}}", from=1-1, to=1-3]
      \end{tikzcd}.
    \end{equation*}
  Now, to find the Empire loyalists, we can posit a global element and, given 
  sufficient structure in the olog, form a pullback in the base of the 
  source-target fibration:
    \begin{equation*}
      \begin{tikzcd}
        {\fbox{Individual}} && {\fbox{Individual}} \\
        {\fbox{Individual}} && {\fbox{Individual}}
        \arrow["1", from=1-1, to=1-3]
        \arrow["1"', from=1-1, to=2-1]
        \arrow["1", from=1-3, to=2-3]
        \arrow[""{name=0, anchor=center, inner sep=0}, "1"', from=2-1, to=2-3]
        \arrow["\lrcorner"{anchor=center, pos=0.125}, draw=none, from=1-1, to=0]
      \end{tikzcd}\qquad\qquad
      \begin{tikzcd}
        {\fbox{Imperial Clan}} && {\fbox{Clan}} \\
        1 && {\fbox{Faction}}
        \arrow["{\text{is a}}", from=1-1, to=1-3]
        \arrow["{!}"', from=1-1, to=2-1]
        \arrow["{\text{Loyalty}}", from=1-3, to=2-3]
        \arrow[""{name=0, anchor=center, inner sep=0}, "{\text{Imperial}}"', from=2-1, to=2-3]
        \arrow["\lrcorner"{anchor=center, pos=0.125}, draw=none, from=1-1, to=0]
      \end{tikzcd}
    \end{equation*}
  Now, there are two ways to do the Loyalist query using Beck-Chevalley. The 
  naive way to start is to abstract Membership first along $(1,\text{Loyalty})$ 
  and then filter for the Imperials as in 
    \begin{equation*}
      \begin{tikzcd}
        {\fbox{Individual}} && {\fbox{Clan}} \\
        {\fbox{Individual}} && {\fbox{Faction}}
        \arrow[""{name=0, anchor=center, inner sep=0}, "{\text{Membership}}"{inner sep=.8ex}, "\shortmid"{marking}, from=1-1, to=1-3]
        \arrow[equals, from=1-1, to=2-1]
        \arrow["{\text{Loyalty}}", from=1-3, to=2-3]
        \arrow[""{name=1, anchor=center, inner sep=0}, "{\text{Loyalty}}"'{inner sep=.8ex}, "\shortmid"{marking}, from=2-1, to=2-3]
        \arrow["{\mathrm{ext}}"{description}, draw=none, from=0, to=1]
      \end{tikzcd}\qquad\qquad
      \begin{tikzcd}
        {\fbox{Individual}} && 1 \\
        {\fbox{Individual}} && {\fbox{Faction}}
        \arrow[""{name=0, anchor=center, inner sep=0}, "{\text{Imperial Loyalists}}"{inner sep=.8ex}, "\shortmid"{marking}, from=1-1, to=1-3]
        \arrow[equals, from=1-1, to=2-1]
        \arrow["{\text{Imperial}}", from=1-3, to=2-3]
        \arrow[""{name=1, anchor=center, inner sep=0}, "{\text{Loyalty}}"'{inner sep=.8ex}, "\shortmid"{marking}, from=2-1, to=2-3]
        \arrow["{\mathrm{res}}"{description}, draw=none, from=0, to=1]
      \end{tikzcd}.
    \end{equation*}
  This has the resulting tables
    \begin{equation*}
      \begin{tabular}{| l | l | }
          \hline\multicolumn{2}{| c |}{\bf Loyalty}\\
          \hline {\bf Individual }&{\bf Faction}\\
          \hline  Asgeir & Stormcloak  \\
          \hline  Eorlund & Stormcloak  \\
          \hline  Galmar & Stormcloak  \\
          \hline  Idgrod & Imperial  \\
          \hline  Nilsine & Imperial \\
          \hline  Torbjorn & Imperial  \\
          \hline  Vignar & Stormcloack  \\
          \hline
      \end{tabular} \qquad\qquad
      \begin{tabular}{| l | }
          \hline\multicolumn{1}{| c |}{\bf Imperial Loyalists}\\
          \hline  Idgrod \\
          \hline  Nilsine\\
          \hline  Torbjorn \\
          \hline
      \end{tabular}.
    \end{equation*} 
  Basically, we looked to see which faction each clan supports, associated each 
  individual to the appropriate faction and then took just the names. This is 
  possibly the straightforward and naive way one might proceed at first. On the 
  other hand, we can do a filter on the Membership relation by $(1,\text{is a})$ 
  first and then select individuals:
    
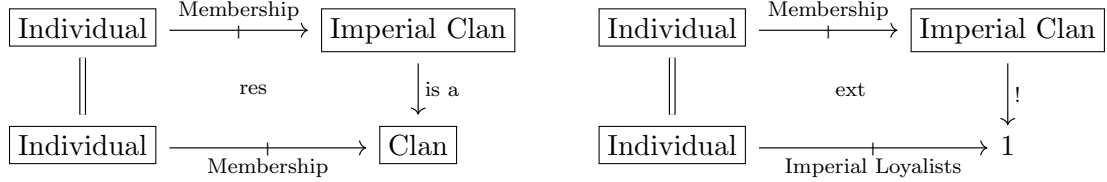
\begin{equation*}
      \begin{tikzcd}
        {\fbox{Individual}} && {\fbox{Imperial Clan}} \\
        {\fbox{Individual}} && {\fbox{Clan}}
        \arrow[""{name=0, anchor=center, inner sep=0}, "{\text{Membership}}"{inner sep=.8ex}, "\shortmid"{marking}, from=1-1, to=1-3]
        \arrow[equals, from=1-1, to=2-1]
        \arrow["{\text{is a}}", from=1-3, to=2-3]
        \arrow[""{name=1, anchor=center, inner sep=0}, "{\text{Membership}}"'{inner sep=.8ex}, "\shortmid"{marking}, from=2-1, to=2-3]
        \arrow["{\mathrm{res}}"{description}, draw=none, from=0, to=1]
      \end{tikzcd}\qquad
      \begin{tikzcd}
        {\fbox{Individual}} && {\fbox{Imperial Clan}} \\
        {\fbox{Individual}} && 1
        \arrow[""{name=0, anchor=center, inner sep=0}, "{\text{Membership}}"{inner sep=.8ex}, "\shortmid"{marking}, from=1-1, to=1-3]
        \arrow[equals, from=1-1, to=2-1]
        \arrow["{!}", from=1-3, to=2-3]
        \arrow[""{name=1, anchor=center, inner sep=0}, "{\text{Imperial Loyalists}}"'{inner sep=.8ex}, "\shortmid"{marking}, from=2-1, to=2-3]
        \arrow["{\mathrm{ext}}"{description}, draw=none, from=0, to=1]
      \end{tikzcd}.
    \end{equation*}
  This is to filter for individuals with their Imperial-supporting clan and then 
  to project to the individuals:
    \begin{equation*}
      \begin{tabular}{| l | l | }
          \hline\multicolumn{2}{| c |}{\bf Membership}\\
          \hline {\bf Individual }&{\bf Imperial Clan}\\
          \hline  Idgrod & Battle-Born  \\
          \hline  Nilsine & Shatter-Shield \\
          \hline  Torbjorn & Shatter-Shield  \\
          \hline
      \end{tabular} \qquad\qquad
      \begin{tabular}{| l | }
          \hline\multicolumn{1}{| c |}{\bf Imperial Loyalists}\\
          \hline  Idgrod \\
          \hline  Nilsine\\
          \hline  Torbjorn \\
          \hline
      \end{tabular}.
    \end{equation*}
  Notice that the latter ends up with a much smaller intermediate table by 
  filtering first. And there is actually no aggregation or abstraction, just a 
  select query by taking the projection. This is basically a free query from a 
  computational expense standpoint. These computations are of course carried out 
  in $\Rel$ where we have Beck-Chevalley and pullbacks in the underlying 
  category $\Set$. The point of axiomatizing Beck-Chevalley (perhaps through the 
  assumption of strong tabulators) in the olog is that the structure-preserving 
  data instance will, by functorial and fibrational semantics, can execute the 
  query by either path allowed by the rewrite rule.
\end{exa}

To conclude, briefly we note that coproducts satisfying Beck-Chevalley 
automatically gives an interpretation of fibered equality in the usual 
fibrational semantics of equational type theories.

\begin{cor} \label{cor:equality}
  Any `double category of relations' for which dependent coproducts satisfy 
  Beck-Chevalley is a Eq-fibration in the sense of 
  \cite[\S\S 3.4-3.5]{jacobs1999}.
\end{cor}
\begin{proof}
  Such a double category always has all dependent coproducts from the equipment 
  structure. Hence it has fibered equality as in \cite[\S 3.4]{jacobs1999} since 
  Beck-Chevalley follows as a special case. Any `double category of relations' 
  has local products (see \cref{section:local-products-exponentiability}).
\end{proof}

\begin{rem}
  A native treatment of equality is somewhat beside the point in the present 
  work as relational database equality is closer to comparing common columns 
  along diagonals, that is, basically a join operation (already studied in 
  \cite{lambert2025}). Equality here is substitution of terms into fibered 
  equality predicates derived from the dependent coproducts applied to local 
  terminals.
\end{rem}

\section{Modality \& Tabulators}
\label{section:modality}

We are left with accounting for the modal queries in 
\cref{fig:quant-queries}. Viewing $r\colon x\proto y$ as a fixed 
proarrow, or perhaps a fixed relational attribute in an olog, we 
have the operation
  \begin{equation} \label{eqn:possibility-operator}
    \diamondsuit_r:=r\odot - \colon \dbl D(y,1)\to\dbl D(x,1) \qquad\qquad p\colon y\proto 1 \quad \mapsto \quad r\odot p
  \end{equation}
by precomposition with $r$. This is a functor by interchange. And 
we drop the subscripted `$r$' whenever convenient. The point of the previous 
discussion and \cref{prop:possibility-rewrite-rule-1} is that a subtype of $y$ 
induces an effect / proposition $y\proto 1$ and then loose 
composition carries the force of existential quantification over 
the subtype thus realizing the modal operator.

\begin{defi}[Adjoint Modality] \label{def:adjoint-modality}
  The operator $\diamondsuit_r$ of \cref{eqn:possibility-operator}
  is the left adjoint of an \textbf{adjoint modality} if it has a 
  right adjoint 
    \begin{equation*}
      \Box\colon\dbl D(x,1)\to\dbl D(y,1) \qquad \diamondsuit_r\dashv\Box
    \end{equation*}
  called \emph{(reverse) necessity}. The pair is referred to as an adjoint modality where $\Box$ is the right adjoint of the modality.
\end{defi}

The reason for the qualifier \emph{reverse} will be made clear below. Notice 
that we have not constructed $\Box$ at any point so far in an elementary way in 
a general cartesian equipment. The double category of relations, however, 
certainly has such adjoint modalities for each fixed relation. Under the 
definition above, these operators justly bear the standard notations not only on 
account of the $\Rel$-semantics. There is also the following observation, 
namely, that these operators satisfy the usual monad and comonad laws, as well 
as the characteristic (B) axiom of modal S5. In this sense, the definition is 
perhaps to strong, since S5 is a relatively refined system; however, our 
construction of the necessity operators using dependent products will satisfy 
these laws anyway.

\begin{prop} \label{prop:modal-properties}
  Let $r\colon x\proto y$ induce an adjoint modality 
  (\cref{def:adjoint-modality}). If $r$ is a preorder, then the possibility and 
  necessity operators associated to $r$ satisfy the monad and comonad axioms 
    \begin{enumerate}
      \item $\id\leq\diamondsuit$ and $\diamondsuit\diamondsuit\leq\diamondsuit$
      \item $\Box\leq \id$ and $\Box\Box\leq\Box$
    \end{enumerate}
  as well as the further (B) axioms, namely, 
    \begin{enumerate}
      \item $\id\leq \Box\diamondsuit$
      \item $\diamondsuit\Box\leq \id$.
    \end{enumerate}
\end{prop}
\begin{proof}
  Since $\diamondsuit\dashv\Box$, it suffices to show that $\diamondsuit$ 
  satisfies the monad axioms as $\Box$ will automatically be a comonad. Likewise 
  the (B) axioms are just a restatement of the adjunction assumption. 
  Reflexivity of $r$ means that there is a canonical cell 
    \begin{equation*}
      \begin{tikzcd}
        x & x \\
        x & x
        \arrow[""{name=0, anchor=center, inner sep=0}, "{\id_x}"{inner sep=.8ex}, "\shortmid"{marking}, from=1-1, to=1-2]
        \arrow[equals, from=1-1, to=2-1]
        \arrow[equals, from=1-2, to=2-2]
        \arrow[""{name=1, anchor=center, inner sep=0}, "r"'{inner sep=.8ex}, "\shortmid"{marking}, from=2-1, to=2-2]
        \arrow["{\delta_x}"{description}, draw=none, from=0, to=1]
      \end{tikzcd}
    \end{equation*}
  expressing the fact that the diagonal factors through $r$. But then composing 
  we have 
    \begin{equation*}
      \begin{tikzcd}
        1 & x & x \\
        1 & x & x
        \arrow[""{name=0, anchor=center, inner sep=0}, "p"{inner sep=.8ex}, "\shortmid"{marking}, from=1-1, to=1-2]
        \arrow[equals, from=1-1, to=2-1]
        \arrow[""{name=1, anchor=center, inner sep=0}, "\id"{inner sep=.8ex}, "\shortmid"{marking}, from=1-2, to=1-3]
        \arrow[equals, from=1-2, to=2-2]
        \arrow[equals, from=1-3, to=2-3]
        \arrow[""{name=2, anchor=center, inner sep=0}, "p"'{inner sep=.8ex}, "\shortmid"{marking}, from=2-1, to=2-2]
        \arrow[""{name=3, anchor=center, inner sep=0}, "r"'{inner sep=.8ex}, "\shortmid"{marking}, from=2-2, to=2-3]
        \arrow["{1_p}"{description}, draw=none, from=0, to=2]
        \arrow["\leq"{description}, draw=none, from=1, to=3]
      \end{tikzcd}
    \end{equation*}
  showing that $p\leq \diamondsuit p$ by loose composition and 
  \cref{prop:possibility-rewrite-rule-1}. Likewise, using reflexivity and the 
  same proposition, we have 
    \begin{equation*}
      \begin{tikzcd}
        1 & x & x & x \\
        1 & x && x
        \arrow[""{name=0, anchor=center, inner sep=0}, "p"{inner sep=.8ex}, "\shortmid"{marking}, from=1-1, to=1-2]
        \arrow[equals, from=1-1, to=2-1]
        \arrow["r"{inner sep=.8ex}, "\shortmid"{marking}, from=1-2, to=1-3]
        \arrow[equals, from=1-2, to=2-2]
        \arrow["r"{inner sep=.8ex}, "\shortmid"{marking}, from=1-3, to=1-4]
        \arrow[equals, from=1-4, to=2-4]
        \arrow[""{name=1, anchor=center, inner sep=0}, "p"'{inner sep=.8ex}, "\shortmid"{marking}, from=2-1, to=2-2]
        \arrow[""{name=2, anchor=center, inner sep=0}, "r"'{inner sep=.8ex}, "\shortmid"{marking}, from=2-2, to=2-4]
        \arrow["{1_p}"{description}, draw=none, from=0, to=1]
        \arrow["\leq"{description}, draw=none, from=1-3, to=2]
      \end{tikzcd}
    \end{equation*}
  so that $\diamondsuit \diamondsuit p \leq \diamondsuit p$ holds, as required.
\end{proof}

\begin{rem}
  There are pecularities of this set-up. First, although classical set-theoretic 
  Kripke semantics require the underlying relation to be not only a preorder but 
  also \emph{symmetric} to ensure that the characteristic (B) axioms of modal S5 
  are satisfied. In fact, equivalence relations are precisely those 
  accessibility relations modeling classical S5. In the present instance, the 
  assumption of the existence of the adjoint modality of course just axiomatizes 
  this axiom scheme, so it suffices to ask that the underlying relation is a 
  preorder to get the characteristic monad and comonad laws. Now, this 
  adjointness is warranted as although classical possibility and necessity are 
  interdefinable using negation, no such relationship exists in general 
  categorical semantics due to the general failure of double negation. 
  Essentially, that is, we ask for some kind of governing interaction between 
  the two operators, and those laws of S5 in the form of the (B) axioms are the 
  clearest, cleanest and most intuitive and well-established. Whereas on the 
  other hand, the various proposed weak interactions axiomatized for 
  intuitionistic S4 are not universally agreed upon and do not imply 
  interdefinability anyway. Second, the tabulator models of necessity in the 
  next section arise from such adjunctions anyway and will thus satisfy the (B) 
  axiom scheme.
\end{rem}

Since possibility can be viewed as loose composition, the modal necessity 
operator could certainly just be asked for in the form of something like 
\emph{closedness} \cite[\S 5]{shulman2008}. Our interest, however, is in 
recovering necessity from the assumed structure of dependent products, since, as 
we have seen, necessity is universal quantification over accessible states. 
\cref{prop:possibility-rewrite-rule-1} shows that possibility is a composite 
operation, namely, a substitution followed by a dependent coproduct further 
confirming that a right adjoint to subsitution would then plausibly recover the 
necessity operator too. That this is so in $\Rel$ can easily be calculated by 
hand. This calculation suggests the general case. In fact, this perspective has 
already been considered in \cite[\S 3]{hermida2011} where the modal operators 
arise as composite substitutions followed by quantification functors. This 
approach is both bicategorical in nature and starts from the bicategories of 
spans or of relations where every proarrow is essentially its own tabulator. Our 
approach applies the same general philosophy to double categories in the 
following way. First recall the standard notion \cite{grandis1999}.

\begin{defi} \label{def:tabulator}
  A double category $\dbl D$ \textbf{has tabulators} if the external identity 
  functor has a right adjoint 
    \begin{equation*}
      \begin{tikzcd}
        {\dbl D_0} & {\dbl{D}_1}
        \arrow["{\id_{-}}", shift left=2, from=1-1, to=1-2]
        \arrow["T", shift left=2, from=1-2, to=1-1]
      \end{tikzcd} \qquad \qquad \id_{-}\dashv T.
    \end{equation*}
  The \textbf{tabulator} of a proarrow $m\colon x\proto y$ is a cell $\tau_m$ 
  having the universal property that there exists a unique morphism $h$ making 
  an equality of diagrams
    \begin{equation*}
      \begin{tikzcd}
        Tm & Tm \\
        x & y
        \arrow[""{name=0, anchor=center, inner sep=0}, "\shortmid"{marking}, equals, from=1-1, to=1-2]
        \arrow["d"', from=1-1, to=2-1]
        \arrow["c", from=1-2, to=2-2]
        \arrow[""{name=1, anchor=center, inner sep=0}, "m"'{inner sep=.8ex}, "\shortmid"{marking}, from=2-1, to=2-2]
        \arrow["{\tau_m}"{description}, draw=none, from=0, to=1]
      \end{tikzcd} \qquad = \qquad 
      \begin{tikzcd}
        Tm & Tm \\
        z & z \\
        x & y
        \arrow[""{name=0, anchor=center, inner sep=0}, "\shortmid"{marking}, equals, from=1-1, to=1-2]
        \arrow["h"', from=1-1, to=2-1]
        \arrow["h", from=1-2, to=2-2]
        \arrow[""{name=1, anchor=center, inner sep=0}, "\shortmid"{marking}, equals, from=2-1, to=2-2]
        \arrow["f"', from=2-1, to=3-1]
        \arrow["g", from=2-2, to=3-2]
        \arrow[""{name=2, anchor=center, inner sep=0}, "m"'{inner sep=.8ex}, "\shortmid"{marking}, from=3-1, to=3-2]
        \arrow["{\id_h}"{description}, draw=none, from=0, to=1]
        \arrow["\theta"{description}, draw=none, from=1, to=2]
      \end{tikzcd}
    \end{equation*}
  for each such cell $\theta$.
\end{defi}

Tabulators viewed as generalized elements constructions ought always to 
roundtrip in one direction, namely, that from proarrows to spans and then back 
to proarrows. This is to say that the elements construction of the base-valued 
functor representing the original fibration object ought to return the original 
fibration object. Note that the other roundtrip works in the case of ordinary 
fibrations and discrete fibrations but fails in \emph{2-toposes} 
\cite{weber2007} due to issues of \emph{size} and \emph{admissibility}. For this 
reason we only axiomatize the former direction. The definition formalizing this 
is then the following one \cite[\S 5.1]{aleiferi2018}.

\begin{defi} \label{def:strong-tabulator}
  The tabulator of a proarrow $p\colon X\proto Y$ is \textbf{strong} if the 
  canonical morphism 
    \begin{equation*}
      \begin{tikzcd}
        Tp & Tp \\
        x & y
        \arrow[""{name=0, anchor=center, inner sep=0}, "\shortmid"{marking}, equals, from=1-1, to=1-2]
        \arrow["d"', from=1-1, to=2-1]
        \arrow["c", from=1-2, to=2-2]
        \arrow[""{name=1, anchor=center, inner sep=0}, "p"'{inner sep=.8ex}, "\shortmid"{marking}, from=2-1, to=2-2]
        \arrow["{\tau_p}"{description}, draw=none, from=0, to=1]
      \end{tikzcd} \qquad = \qquad
      \begin{tikzcd}
        Tp & Tp \\
        x & y \\
        x & y
        \arrow[""{name=0, anchor=center, inner sep=0}, "\shortmid"{marking}, equals, from=1-1, to=1-2]
        \arrow["d"', from=1-1, to=2-1]
        \arrow["c", from=1-2, to=2-2]
        \arrow[""{name=1, anchor=center, inner sep=0}, "{d^*\odot c_!}"'{inner sep=.8ex}, "\shortmid"{marking}, from=2-1, to=2-2]
        \arrow[equals, from=2-1, to=3-1]
        \arrow[equals, from=2-2, to=3-2]
        \arrow[""{name=2, anchor=center, inner sep=0}, "p"'{inner sep=.8ex}, "\shortmid"{marking}, from=3-1, to=3-2]
        \arrow["{\tau_p}"{description}, draw=none, from=0, to=1]
        \arrow["{\exists\,!}"{description, pos=0.6}, draw=none, from=1, to=2]
      \end{tikzcd}
    \end{equation*}
    is an isomorphism. A double category has \textbf{strong tabulators} if it 
    has tabulators and each is strong.
\end{defi}

\begin{rem}
  This means, equivalently, that each such tabulator cell $\tau_p$ is an 
  extension cell. Consequently, a double category has strong tabulators if each 
  of its proarrows is the extension of its tabulator. The double category of 
  relations is an example.
\end{rem}

\begin{lem}
  For any such strong tabulator, the three constructions of $\diamondsuit \phi$ 
  coincide in the sense that
    \begin{equation*}
      \coprod_{d}c^*\phi \cong r\odot\phi 
    \end{equation*}
  holds for any proposition $\phi\colon y\proto 1$ and $r\colon x\proto y$ whose 
  tabulator is strong.
\end{lem}
\begin{proof}
  Consider 
    \begin{equation*}
      \coprod_{d}c^*\phi \cong d^*\odot c_!\odot \phi \cong r\odot\phi
    \end{equation*}
  by the definition of the action $c^*(-) = c_!\odot -$ and that $r$ has a 
  strong tabulator. If $\phi$ arises as the extension of a subtype or other 
  arrow into $y$, apply \cref{prop:possibility-rewrite-rule-1} to see that the 
  other possible definition agrees with these two.
\end{proof}

\begin{defi} \label{def:necessity-from-dep-products}
  Let $\dbl D$ denote a cartesian equipment with strong tabulators and dependent 
  products. The \textbf{modal necessity} operator associated to 
  $r\colon x\proto y$ is the composite functor 
    \begin{equation*}
      \Box_r:=\prod_{d}c^*(-)\colon\dbl D(y,1)\to\dbl D(x,1)
    \end{equation*}
  where $\langle d,c\rangle\colon Tr\to x\times y$ is the tabulator of $r$.
\end{defi}

\begin{rem}
  \cref{def:necessity-from-dep-products} is to priviledge modalities in the 
  \emph{same direction} as in \cite[Remark 3.3]{hermida2011}. That is, both 
  necessity and possibility for a given $r$ return values of its domain. In the 
  measurement interpretation, this is to return the states for which all 
  measurements fall inside the \emph{good} values specified by the proposition 
  in question. Now, in any such double category, there is a reverse operation 
  $r\mapsto r^\dagger$ by taking an extension along the legs of the tabulator of 
  $r$ in reversed order. This coincides with the operation $r\mapsto r^\circ$ 
  from \cite{carboni1987} as discussed above in good cases \cite{lambert2022}. 
  Thus, there are actually four operators
    \begin{enumerate}
      \item $\Box_r:=\prod_{d}c^*(-)$
      \item $\Box_{r^\dagger}:=\prod_{c}d^*(-)$
      \item $\diamondsuit_r:=\coprod_dc^*(-)$
      \item $\diamondsuit_{r^\dagger}:=\coprod_cd^*(-)$
    \end{enumerate}
  This four-fold modal set-up is a feature of such relational and hyperdoctrine 
  accounts of modality. The \emph{up-modalities} are the first and third above, 
  returning states where outcomes stay within the specified allowed values; 
  whereas the \emph{down-modalities} are the second and fourth, returing the 
  outcomes whose states are among the specified ones. When $r$ is an 
  endorelation, up- or down-modalities simply specify subsequent or antecedent 
  states. In any case, the two types are inter-related in the following way.
\end{rem}

\begin{lem}
  Under the hypotheses of \cref{def:necessity-from-dep-products}, the (B) axioms 
  are satisfied and the pairs $\diamondsuit_r\dashv\Box_{r^\dagger}$ and 
  $\diamondsuit_{r^\dagger}\dashv \Box_r$ are adjoint modalities as in 
  \cref{def:adjoint-modality}.  If $r\colon x\proto x$ is a preorder, then all 
  the (co)monad axioms for both up- and down-modalities are satisfied.
\end{lem}
\begin{proof}
  See \cref{prop:modal-properties} for the last part. For the adjoint modality 
  statement, consider the diagram of adjunctions
    \begin{equation*}
      \begin{tikzcd}
        {\dbl D(y,1)} & {\dbl D(Tr,1)} & {\dbl D(x,1)}
        \arrow["{c^*(-)}", shift left=2, from=1-1, to=1-2]
        \arrow["{\prod_c}", shift left=2, from=1-2, to=1-1]
        \arrow["{\coprod_d}", shift left=2, from=1-2, to=1-3]
        \arrow["{d^*(-)}", shift left=2, from=1-3, to=1-2]
      \end{tikzcd}
    \end{equation*}
  showing that $\diamondsuit_r\dashv\Box_{r^\dagger}$, as required. The other 
  statement is dual.
\end{proof}

\begin{exa}
  When specialized to $\Rel$, the necessity operator here returns the necessity 
  query in \cref{fig:quant-queries}, thus completing our account of those six 
  queries.
\end{exa}

\begin{rem}
  Evidently symmetry in the form $r^\dagger\leq r$ collapses the distinction 
  between the possibility and necessity pairs: 
  $\diamondsuit_r = \diamondsuit_{r^\dagger}$ and $\Box_r = \Box_{r^\dagger}$ so 
  that directly $\diamondsuit_r\dashv\Box_r$. In this way, an equivalence 
  relation $r$ induces an adjoint modality between the now just \emph{two} modal 
  operators without the foregoing up- vs. down- distinction.
\end{rem}

\part{Cartesian Closedness \& FOML}

This second part concerns ultimately the way in which double ologs model 
first-order modal logic (FOML) locally. Inasmuch as the classical operations of 
relational databases are framable in FOML this means that suitably structured 
double ologs thus provide a complete synthetic framework for database querying. 
There are two technical aspects of this development. First is to connect 
cartesian closedness with the assumed dependent products coming with any 
$\prod$-double olog. The connection is made via the additional structure of 
tabulators. To show that this works, we study in some detail the modular laws 
and Frobenius reciprocity which are valid in any `double category of relations'. 
Our analysis of local products with identity proarrows leads to directly to our 
formulas for the general exponential. Secondly, the evident utility of a 
hypothetical negation operator leads to the natural question of cocartesian 
structure and whether, how and to what extent this results in local disjunction 
and how it interacts with local products. It turns out that, assuming a global 
distributive law, conjunction and disjunction distribute locally as well. These 
propositional and first-order laws are then compared with the various canonical 
modal statements such as the usual (K) axiom. Finally, it is shown that such 
first-order models interpret description logic.

\section{Local Products \& Exponentiability}
\label{section:local-products-exponentiability}

To continue the development, we work first with the given cartesian structure 
assumed in our ambient set-ups without assuming further structures as in later 
sections. If $\dbl D$ is a cartesian equipment, then each hom-category $\dbl D(x,y)$ has cartesian products, given by the formula 
  \begin{equation*} \label{eqn:local-products-construction}
    \begin{tikzcd}
      x & y \\
      1 & 1
      \arrow[""{name=0, anchor=center, inner sep=0}, "{1_{x,y}}"{inner sep=.8ex}, "\shortmid"{marking}, from=1-1, to=1-2]
      \arrow["{!}"', from=1-1, to=2-1]
      \arrow["{!}", from=1-2, to=2-2]
      \arrow[""{name=1, anchor=center, inner sep=0}, "\shortmid"{marking}, from=2-1, to=2-2]
      \arrow["{\mathrm{res}}"{description}, draw=none, from=0, to=1]
    \end{tikzcd}\qquad\qquad
    \begin{tikzcd}
      x && y \\
      {x\times x} && {y\times y}
      \arrow[""{name=0, anchor=center, inner sep=0}, "{m\wedge n}"{inner sep=.8ex}, "\shortmid"{marking}, from=1-1, to=1-3]
      \arrow["{\Delta_x}"', from=1-1, to=2-1]
      \arrow["{\Delta_y}", from=1-3, to=2-3]
      \arrow[""{name=1, anchor=center, inner sep=0}, "{m\times n}"'{inner sep=.8ex}, "\shortmid"{marking}, from=2-1, to=2-3]
      \arrow["{\mathrm{res}}"{description}, draw=none, from=0, to=1]
    \end{tikzcd}
  \end{equation*}
as described in \cite[Proposition 4.3.2]{aleiferi2018}. Note that the wedge 
operator is thus understood also as a loose composite 
  \begin{equation*}
    m\wedge n \cong \Delta_!\odot (m\times n) \odot \Delta^*
  \end{equation*}
by the usual construction of restrictions. This leads 
naturally to the following definitions. 

\begin{defi}\label{def:exponentiable-and-cartesianclosed}
  A proarrow $m$ in a cartesian equipment $\dbl D$ is \textbf{exponentiable} if 
  each product functor 
    \begin{equation*}
      m\wedge-\colon \dbl{D}(x,y)\to\dbl{D}(x,y)
    \end{equation*}
  has a right adjoint $(-)^m$. The object $n^m$ is called the 
  \textbf{exponential} of $n$ by $m$. A cartesian equipment is \textbf{locally 
  cartesian closed} if each proarrow is exponentiable.
\end{defi}

\begin{exa}
  For relations $R,S\colon X\proto Y$, the exponential is given by 
    \begin{equation*}
      S^R = \lbrace (x,y)\mid xRy \text{ implies } xSy\rbrace
    \end{equation*}
  and computed via adjoints as 
    \begin{equation*}
      \forall_RR^*(S)
    \end{equation*}
  thinking of $R$ as its own tabulator.
\end{exa}

Of course we may simply ask that our cartesian equipments are locally cartesian 
closed in this sense. It is a point of interest, however, to connect this 
structure to the right adjoints to substitution assumed in previous sections 
inasmuch as, (1) these are supposed to provide \emph{generalized hom objects} 
and (2) there is an evident universal quantification over pairs $(x,y)$ in the 
example above. The answer, roughly, is that these right adjoints are computable 
from right adjoints to substitution whenever the given cartesian equipment has 
strong tabulators. This development showcases the relationships between 
compactness, Frobenius, modularity, comprehension schemes, and exponentiability 
all in terms of the requested equipment and cartesian structure on the ambient 
double category.

\section{Modularity \& Frobenius}
\label{section:modularity-frobenius}

It is asserted in \cite[\S 3]{lambert2022} and proved in the cited references 
that any `double category of relations' satisfies the \emph{Frobenius 
reciprocity Law} for the accompanying doctrine. This is proved here in detail 
since the proof is interesting in its own right and the law is the main 
ingredient in proving that right adjoints to substitution give cartesian closed 
structure.

Recall from \cite{carboni1987} that any `bicategory of relations' is compact 
closed and has an involution operation 
  \begin{equation*}
    f \mapsto f^\circ
  \end{equation*}
satisfying 
  \begin{equation*}
    1^\circ = 1\qquad (f^\circ)^\circ = f\qquad (gf)^\circ = f^\circ g^\circ
  \end{equation*}
with $(-)^\circ$ defined using the unit and counit in the compact closed 
structure as in the proof of \cite[Theorem 2.4]{carboni1987}. Now, if $f$ is a 
\emph{map} then $f^\circ$ is that right adjoint $f\dashv f^\circ$ 
\cite[Lemma 2.5]{carboni1987}. This extends to `double categories of relations' 
in the following way.

\begin{lem}
  In any `double category of relations', the equations
    \begin{enumerate}
      \item $(f_!)^\circ = f^*$
      \item $(f^*)^\circ = f_!$
    \end{enumerate}
  each hold.
\end{lem}
\begin{proof}
  Start with the usual equipment adjunction $f_!\dashv f^*$ 
  \cite[\S 5]{shulman2008}. The underlying bicategory of any `double category of 
  relations' is a `bicategory of relations' \cite{lambert2022} and bicategorical 
  adjoints are unique up to isomorphism.
\end{proof}

The proof of the modular law (cf. \cite{freyd1990}) is now an application. It is 
phrased here in the manner needed in proving Frobenius.

\begin{prop}[modular laws] \label{prop:modular-laws}
  In any `double category of relations', the two inequalities 
    \begin{enumerate}
      \item $(f^*\odot m)\wedge n \leq f^*\odot (m\wedge f_!\odot n)$
      \item $(p\odot f_!)\wedge q \leq (p\wedge q\odot f^*)\odot f_!$
    \end{enumerate}
  each hold.
\end{prop}
\begin{proof}
  Two inequalities follow directly from \cite[Theorem 2.4]{carboni1987}. These 
  are 
    \begin{enumerate}
      \item $\Delta_!\odot (f^*\times 1) \leq f^*\odot\Delta_!\odot (1\times f_!)$
      \item $(f_!\times 1)\odot \Delta^*\leq (1\times f^*)\odot \Delta^*\odot f_!$
    \end{enumerate}
  using $(f^*)^\circ = f_!$ as in the lemma above for the first. The two modular 
  laws are direct applications. For the first one, we have 
    \begin{align*}
      (f^*\odot m)\wedge n &= \Delta_!\odot (f^*\odot m\times n)\odot\Delta^* \\
                           &= \Delta_!\odot (f^*\times 1\odot m\times n)\odot\Delta^* \\
                           &\leq f^*\odot\Delta_!\odot (1\times f_! \odot m\times n)\odot\Delta^* \\
                           &= f^*\odot\Delta_!\odot (m\times f_! \odot n)\odot\Delta^* \\
                           &= f^*\odot (m\wedge f_!\odot n)
    \end{align*}
  using the interaction of $\odot$ and $\times$ as well as the definition of 
  local products. The other inequality is analogous.
\end{proof}

\begin{con}[Frobenius Comparison Cell]\label{construction:Frob-comparison-cell}
  Let $\dbl D$ denote any cartesian equipment. Local products make sense in this 
  context. Let $f\colon x\to z$ and $g\colon y\to w$ denote morphisms. First 
  form the coproduct extension as the third step in the sequence:
    \begin{equation*}
      \begin{tikzcd}
        x & y \\
        z & w
        \arrow[""{name=0, anchor=center, inner sep=0}, "{f_!pg^*}"{inner sep=.8ex}, "\shortmid"{marking}, from=1-1, to=1-2]
        \arrow["f"', from=1-1, to=2-1]
        \arrow["g", from=1-2, to=2-2]
        \arrow[""{name=1, anchor=center, inner sep=0}, "p"'{inner sep=.8ex}, "\shortmid"{marking}, from=2-1, to=2-2]
        \arrow["{\mathrm{res}}"{description}, draw=none, from=0, to=1]
      \end{tikzcd}\quad\leadsto\quad
      \begin{tikzcd}
        x && y \\
        {x\times x} && {y\times y}
        \arrow[""{name=0, anchor=center, inner sep=0}, "{m\wedge f_!pg^*}"{inner sep=.8ex}, "\shortmid"{marking}, from=1-1, to=1-3]
        \arrow["{\Delta_x}"', from=1-1, to=2-1]
        \arrow["{\Delta_y}", from=1-3, to=2-3]
        \arrow[""{name=1, anchor=center, inner sep=0}, "{m\times f_!pg^*}"'{inner sep=.8ex}, "\shortmid"{marking}, from=2-1, to=2-3]
        \arrow["{\mathrm{res}}"{description}, draw=none, from=0, to=1]
      \end{tikzcd}\quad\leadsto\quad
      \begin{tikzcd}
        x && y \\
        z && w
        \arrow[""{name=0, anchor=center, inner sep=0}, "{m\wedge f_!pg^*}"{inner sep=.8ex}, "\shortmid"{marking}, from=1-1, to=1-3]
        \arrow["f"', from=1-1, to=2-1]
        \arrow["g", from=1-3, to=2-3]
        \arrow[""{name=1, anchor=center, inner sep=0}, "{\coprod_{f,g}(m\wedge f_!pg^*)}"'{inner sep=.8ex}, "\shortmid"{marking}, from=2-1, to=2-3]
        \arrow["{\mathrm{ext}}"{description}, draw=none, from=0, to=1]
      \end{tikzcd}.
    \end{equation*}
  Note that we are suppressing some of the loose composition notation for 
  readability. Now, on the other hand, form the local product as the second step 
  in the sequence:
    \begin{equation*}
      \begin{tikzcd}
        x & y \\
        z & w
        \arrow[""{name=0, anchor=center, inner sep=0}, "m"{inner sep=.8ex}, "\shortmid"{marking}, from=1-1, to=1-2]
        \arrow["f"', from=1-1, to=2-1]
        \arrow["g", from=1-2, to=2-2]
        \arrow[""{name=1, anchor=center, inner sep=0}, "{\coprod_{f,g}m}"'{inner sep=.8ex}, "\shortmid"{marking}, from=2-1, to=2-2]
        \arrow["{\mathrm{ext}}"{description}, draw=none, from=0, to=1]
      \end{tikzcd}\qquad\leadsto\qquad 
      \begin{tikzcd}
        z && w \\
        {z\times z} && {w\times w}
        \arrow[""{name=0, anchor=center, inner sep=0}, "{(\coprod_{f,g}m)\wedge p}"{inner sep=.8ex}, "\shortmid"{marking}, from=1-1, to=1-3]
        \arrow["{\Delta_z}"', from=1-1, to=2-1]
        \arrow["{\Delta_w}", from=1-3, to=2-3]
        \arrow[""{name=1, anchor=center, inner sep=0}, "{(\coprod_{f,g}m)\times p}"'{inner sep=.8ex}, "\shortmid"{marking}, from=2-1, to=2-3]
        \arrow["{\mathrm{res}}"{description}, draw=none, from=0, to=1]
      \end{tikzcd}
    \end{equation*}
  By the universal property of the local product immediately above, there is 
  thus a canonical comparison cell 
    \begin{equation*}
      m\wedge f_!pg^* \Rightarrow (\coprod_{f,g}m)\wedge p
    \end{equation*}
  and likewise a comparison cell 
    \begin{equation*}
      \coprod_{f,g}(m\wedge f_!pg^*) \Rightarrow (\coprod_{f,g}m)\wedge p
    \end{equation*}
  by the universal property of the extension cell defining the coproduct in the 
  preantepenultimate display. This is the required \emph{Frobenius comparison 
  cell}. Note that it exists even if, as assumed, $\dbl D$ is merely a cartesian 
  equipment.
\end{con}

\begin{prop} \label{prop:Frobenius-reciprocity}
  Any `double category of relations' satisfies the Frobenius reciprocity Law. 
  That is, the canonical cell 
    \begin{equation*}
      \coprod_{f,g}(m\wedge f_!pg^*) \leq (\coprod_{f,g}m)\wedge p
    \end{equation*}
  described above is invertible.
\end{prop}
\begin{proof}
  For the proof, we produce an reverse inequality. Uniqueness then forces the 
  inequality above to be an equality. Starting from the left side and using the 
  construction of extensions and both modular laws (\cref{prop:modular-laws}), 
  we have 
    \begin{align*}
      (\coprod_{f,g}m)\wedge p &=(f^*\odot m \odot g_!)\wedge p \\
                               &\leq f^*\odot (m\odot g^*\wedge f_!\odot p) \\
                               &\leq f^*\odot (m \wedge (f_!\odot p\odot g^*))\odot g_! \\
                               &=\coprod_{f,g}(m\wedge f_!pg^*) 
    \end{align*}
  as required.
\end{proof}

\begin{cor}
  Any `double category of relations', viewed as an equipment, is a regular 
  fibration \cite[Definition 4.2.1]{jacobs1999}, hence a fibrational model of 
  regular logic.
\end{cor}
\begin{proof}
  It needs only to be observed that such a double olog has equality in the sense 
  of \cite[\S 3.4]{jacobs1999} since it has \emph{all} coproducts satisfying 
  Frobenius reciprocity by the above. This was noted above in 
  \cref{cor:equality}.
\end{proof}

\begin{rem}[Query Optimization] \label{remark:Frobenius-Optimization}
  Frobenius reciprocity is a \emph{join pushdown} optimization rewrite rule, 
  saying that a extension following an expensive join can be equivalently 
  executed by first collapsing using an extension and then performing a smaller 
  join. This has the effect of restricting to a smaller table via a collapse 
  operation, hence surveying a smaller totality, \emph{before} computing the 
  product. That is, the right hand side of the Frobenius reciprocity identity in 
  \cref{prop:Frobenius-reciprocity} above is the optimized query and the result 
  says it produces the same result. 
\end{rem}

\begin{exa} \label{example:frobenius}
  We illustrate this with a simple example. This scenario is one of 
  \emph{abstraction} and \emph{querying metadata}. We have a surjective function 
  associating to each of several individuals his or her class:
    \begin{equation*}
      \begin{tabular}{| l | l | }
          \hline\multicolumn{2}{| c |}{\bf Belongs to}\\
          \hline {\bf Individual }&{\bf Class}\\
          \hline Belrand & spellsword \\
          \hline Farcas & warrior \\
          \hline Illia & mage \\
          \hline J'zargo & mage \\
          \hline Marcurio & mage \\
          \hline Teldryn Sero & spellsword \\
          \hline
      \end{tabular}. 
    \end{equation*}
  Suppose we have the following subsets
    \begin{equation*}
      S = \lbrace \text{Marcurio}, \text{ Illia}, \text{ Teldryn Sero}\rbrace
      \qquad T = \lbrace \text{spellsword}, \text{ warrior}\rbrace
    \end{equation*}
  and suppose we wish to know which of the classes on the latter list are 
  represented by the individuals on the former list. Of course the example is 
  simple enough that we can just look at the table. But for large datasets a 
  process is of course required. A naive way to proceed is see who on the list 
  falls into any of the specified classes and then just take the class 
  information. That is, more formally, restrict $T$ and intersect with $S$ to 
  see who on the list falls into the classes; and then to project back to the 
  codomain to see which classes we ended up with:
    \begin{align*}
      &\exists_{\text{Belongs}}(S\wedge (\text{Belongs}^*(T)))\\ 
        = \;&\exists_{\text{Belongs}}(\lbrace \text{Marcurio}, \text{ Illia}, \text{ 
        Teldryn Sero}\rbrace\cap \lbrace \text{Belrand}, \text{ Farcas}, \text{ Teldryn Sero}\rbrace )\\
        = \;&\exists_{\text{Belongs}}(\lbrace \text{Teldryn Sero}\rbrace) \\
        = \;&\lbrace \text{spellsword}\rbrace.
    \end{align*}
  Note that three operations are required. The restriction and intersection are 
  where most of the computation is done. But since we are working in relations 
  where Frobenius reciprocity holds, the exact same effect is had by 
  \emph{first} abstracting or aggregating the metadata and \text{then} 
  intersecting:
    \begin{equation*}
      \exists_{\text{Belongs}}(S)\wedge T = \lbrace\text{spellsword}\rbrace\cap\lbrace \text{spellsword},\text{ warrior}\rbrace = \lbrace\text{spellsword}\rbrace.
    \end{equation*}
  Only two operations are required. The point is that Frobenius reciprocity 
  formalizes this relationship for double ologs and the functorial and 
  fibrational semantics of data instances control the computation of the query 
  in the receiving data structure. Notably, the proposition above shows that it 
  emerges as a derivable \emph{property} of a `double categories of relations'.
\end{exa}

\section{Cartesian Closedness via Products \& Tabulators}
\label{section:cartesian-closedness}

The first result concerns exponentiability. This is an argument that local 
products with identity proarrows are essentially restrictions along the relevant 
diagonal. In this sense they are somewhat like identity predicates 
\cite[\S 3.4]{jacobs1999} and in paricular are left adjoints. 

\begin{lem} \label{lemma:identity-proarrow-exponentiable}
  Any identity proarrow in a $\prod$-double olog is exponentiable.
\end{lem}
\begin{proof}
  The compact closed transpose of $\id\wedge p$ to a proarrow 
  $x\times x\proto 1$ is the composite
    \begin{equation*}
      \widehat{p\wedge\id}  = ((p\wedge\id) \times \id) \odot\Delta^*\odot x_!
    \end{equation*}
  as in \cite[Theorem 2.4]{carboni1987}. Now, this can be rewritten into a more 
  convenient form via the following calculation:
    \begin{align*}
      \widehat{p\wedge\id} &= (\Delta_!\times \id)\odot (p\times\id\times\id)\odot(\Delta^*\times\id)\odot\Delta^*\odot x_! \qquad &\text{(\text{Constr. Local Product})}\\
      &=(\Delta_!\times \id)\odot (p\times\id\times\id)\odot(\id\times\Delta^*)\odot\Delta^*\odot x_! \qquad &\text{(\text{Comonoid Axiom})}\\
      &=(\Delta_!\times \id)\odot (\id\times\Delta^*)\odot (p\times\id)\odot\Delta^*\odot x_! \qquad &\text{(\text{Functoriality of $\times$})}\\
      &=\Delta^*\odot \Delta_!\odot (p\times\id) \odot\Delta^*\odot x_! \qquad &\text{(\text{Discreteness})}
    \end{align*}
  which shows that the compact closed transpose of the local product is 
  precisely the following extension following a restriction:
    \begin{equation} \label{eqn:transpose-local-product-with-identity}
      \widehat{p\wedge\id} = \coprod_{\Delta,1} (\Delta,1)^*((p\times \id)\odot\Delta\odot x_!).
    \end{equation}
  Now, using this identity, we have a bijection of cells in the following 
  sequence:
    \begin{align*}
      \begin{tikzcd}[ampersand replacement=\&]
        x \& x \\
        x \& x
        \arrow[""{name=0, anchor=center, inner sep=0}, "{p\wedge\id}"{inner sep=.8ex}, "\shortmid"{marking}, from=1-1, to=1-2]
        \arrow[equals, from=1-1, to=2-1]
        \arrow[equals, from=1-2, to=2-2]
        \arrow[""{name=1, anchor=center, inner sep=0}, "n"'{inner sep=.8ex}, "\shortmid"{marking}, from=2-1, to=2-2]
        \arrow["\leq"{description}, draw=none, from=0, to=1]
      \end{tikzcd} \quad &\leftrightarrow\quad 
      \begin{tikzcd}[ampersand replacement=\&]
        {x\times x} \&\& 1 \\
        {x\times x} \&\& 1
        \arrow[""{name=0, anchor=center, inner sep=0}, "{\widehat{p\wedge\id}}"{inner sep=.8ex}, "\shortmid"{marking}, from=1-1, to=1-3]
        \arrow[equals, from=1-1, to=2-1]
        \arrow[equals, from=1-3, to=2-3]
        \arrow[""{name=1, anchor=center, inner sep=0}, "{\widehat n}"'{inner sep=.8ex}, "\shortmid"{marking}, from=2-1, to=2-3]
        \arrow["\leq"{description}, draw=none, from=0, to=1]
      \end{tikzcd}\qquad &(\text{Compactness})\\
      & = \quad 
      \begin{tikzcd}[ampersand replacement=\&]
        {x\times x} \&\&\&\& 1 \\
        {x\times x} \&\&\&\& 1
        \arrow[""{name=0, anchor=center, inner sep=0}, "{\coprod_{\Delta,1} (\Delta,1)^*((p\times \id)\odot\Delta\odot x_!)}"{inner sep=.8ex}, "\shortmid"{marking}, from=1-1, to=1-5]
        \arrow[equals, from=1-1, to=2-1]
        \arrow[equals, from=1-5, to=2-5]
        \arrow[""{name=1, anchor=center, inner sep=0}, "{\widehat n}"'{inner sep=.8ex}, "\shortmid"{marking}, from=2-1, to=2-5]
        \arrow["\leq"{description}, draw=none, from=0, to=1]
      \end{tikzcd} \qquad &(\text{\cref{eqn:transpose-local-product-with-identity}})\\ 
      &\leftrightarrow \quad
      \begin{tikzcd}[ampersand replacement=\&]
        {x\times x} \&\&\& 1 \\
        {x\times x} \&\&\& 1
        \arrow[""{name=0, anchor=center, inner sep=0}, "{(\Delta,1)^*((p\times \id)\odot\Delta\odot x_!)}"{inner sep=.8ex}, "\shortmid"{marking}, from=1-1, to=1-4]
        \arrow[equals, from=1-1, to=2-1]
        \arrow[equals, from=1-4, to=2-4]
        \arrow[""{name=1, anchor=center, inner sep=0}, "{(\Delta,1)^*(\widehat n)}"'{inner sep=.8ex}, "\shortmid"{marking}, from=2-1, to=2-4]
        \arrow["\leq"{description}, draw=none, from=0, to=1]
      \end{tikzcd} \qquad &(\text{Right Adjoint})\\
      &\leftrightarrow \quad
      \begin{tikzcd}[ampersand replacement=\&]
        {x\times x} \&\&\& 1 \\
        {x\times x} \&\&\& 1
        \arrow[""{name=0, anchor=center, inner sep=0}, "{((p\times \id)\odot\Delta\odot x_!)}"{inner sep=.8ex}, "\shortmid"{marking}, from=1-1, to=1-4]
        \arrow[equals, from=1-1, to=2-1]
        \arrow[equals, from=1-4, to=2-4]
        \arrow[""{name=1, anchor=center, inner sep=0}, "{\prod_{\Delta,1}(\Delta,1)^*(\widehat n)}"'{inner sep=.8ex}, "\shortmid"{marking}, from=2-1, to=2-4]
        \arrow["\leq"{description}, draw=none, from=0, to=1]
      \end{tikzcd} \qquad &(\text{Right Adjoint})\\
      &\leftrightarrow\quad 
      \begin{tikzcd}[ampersand replacement=\&]
        x \&\&\& x \\
        x \&\&\& x
        \arrow[""{name=0, anchor=center, inner sep=0}, "p"{inner sep=.8ex}, "\shortmid"{marking}, from=1-1, to=1-4]
        \arrow[equals, from=1-1, to=2-1]
        \arrow[equals, from=1-4, to=2-4]
        \arrow[""{name=1, anchor=center, inner sep=0}, "{\overline{\prod_{\Delta,1}(\Delta,1)^*(\widehat n)}}"'{inner sep=.8ex}, "\shortmid"{marking}, from=2-1, to=2-4]
        \arrow["\leq"{description}, draw=none, from=0, to=1]
      \end{tikzcd}\qquad &(\text{Compactness})
    \end{align*}
  which proves that $-\wedge\id$ has a right adjoint, so that $\id$ is 
  exponentiable as in \cref{def:exponentiable-and-cartesianclosed}, as claimed.
\end{proof}

Some notation for the main result needs to be established. Use `$\id\supset n$' 
as notation for each right adjoint as in 
\cref{lemma:identity-proarrow-exponentiable}. Thus, we have the adjunction 
statement $p\wedge\id\leq n$ iff $p\leq \id\supset n$. Likewise, for 
$m\colon x\proto y$ and its tabulator $Tm$, we write $m^*(n)$ for the proarrow 
domain of the restriction cell 
  \begin{equation*}
    \begin{tikzcd}
      Tm & Tm \\
      x & y
      \arrow[""{name=0, anchor=center, inner sep=0}, "{m^*(n)}"{inner sep=.8ex}, "\shortmid"{marking}, from=1-1, to=1-2]
      \arrow["d"', from=1-1, to=2-1]
      \arrow["c", from=1-2, to=2-2]
      \arrow[""{name=1, anchor=center, inner sep=0}, "n"'{inner sep=.8ex}, "\shortmid"{marking}, from=2-1, to=2-2]
      \arrow["{\mathrm{res}}"{description}, draw=none, from=0, to=1]
    \end{tikzcd}
  \end{equation*}
Likewise $\prod_m$ and $\Sigma_m$ denote the right- and left-adjoints to 
substitution along the domain and codomain morphisms coming with $Tm$. As in 
\cref{def:tabulator} write $t_m$ for the span formed by the legs of the 
tabulator of $m$. 

\begin{thm} \label{theo:locally-cartesian-closed}
  If $\dbl D$ is a $\prod$-double olog with strong tabulators, then each hom 
  category $\dbl D(x,y)$ is cartesian closed and the formulas
    \begin{equation*}
      m\wedge n = \coprod_m(\id_{Tm}\wedge m^*(n))\qquad\qquad m\supset p = \prod_m(\id_{Tm}\supset m^*(p))
    \end{equation*}
  compute the product and exponential.
\end{thm}
\begin{proof}
  The local product formula is immediate by strength and Frobenius:
    \begin{equation*}
      m\wedge n = (\coprod_m\id_{Tm})\wedge n\cong \coprod_m(\id_{Tm}\wedge m^*(n)).
    \end{equation*}
  Apply \cref{lemma:identity-proarrow-exponentiable} to the local product in the 
  coproduct on the right. This gives:
    \begin{equation*}
      m\wedge n \leq p \qquad \text{iff} \qquad n\leq \prod_m(\id_{Tm}\supset m^*(p))
    \end{equation*}
  as claimed.
\end{proof}

\begin{cor}
  Any $\prod$-double olog with strong tabulators is locally cartesian closed as 
  a fibration (cf. \cite[\S 1.8]{jacobs1999}).
\end{cor}
\begin{proof}
  Any `double category of relations' satisfies Beck-Chevalley and Frobenius 
  reciprocity. The previous shows each fiber is cartesian closed. Reindexing 
  preserves the closed structure by Beck-Chevalley.
\end{proof}

\begin{rem}
  The formulas in \cref{theo:locally-cartesian-closed} closely resemble and in 
  fact are inspired by, or even based upon, those appearing in 
  \cite[Proposition 10.5.4]{jacobs1999}. However, the formulas in the reference 
  appear in the entirely different context of \emph{closed comprehension 
  categories} and provide fiberwise cartesian closed structure as a result of 
  assumptions placed on that structure. Now, an original version of 
  \cref{theo:locally-cartesian-closed} mimicked this development by asking for a 
  highly structured comprehension scheme \cite[\S 8]{lambert2022} via tabulators 
  and valued in spans as in \cite{niefield2012}. The argument of 
  \cite[Proposition 10.5.4]{jacobs1999} runs almost exactly: local products are 
  almost seen to be obtained as a successive application of left adjoints. There 
  is one issue, however, namely, that without in some way inserting a diagonal, 
  the sequence of adjunction bijections does not end up producing the local 
  product in the receiving structure of the comprehension scheme. In the 
  reference this is the crucial step that shows that the left adjoint formula 
  for the local product works. Due to the \emph{split contexts} of proarrows in 
  double categories, this necessary move fails here. What was recognized is that 
  the needed effect can be achieved by intersecting with an identity predicate, 
  hence the local product with identity proarrows considered in 
  \cref{lemma:identity-proarrow-exponentiable}. Now, to make this work, however, 
  the identity proarrows must actually behave like identity predicates. The 
  former technical key for this is the assumption of \emph{unit-purity}. 
  However, this has been made redundant as it is implied by discreteness in 
  `double categories of relations'. Thus, altogether, the result of 
  \cref{theo:locally-cartesian-closed} is that the \emph{native structure} of a 
  $\prod$-double olog with strong tabulators is \emph{always homwise cartesian 
  closed} without the additional immense structuring assumptions on a 
  comprehension scheme. Note also that the formulas in 
  \cref{theo:locally-cartesian-closed} are almost exactly those in 
  \cite[Proposition 10.5.4]{jacobs1999} except on the product side adjusted by 
  interesection with an identity and on the implication side adjusted by the 
  identity antecedent within the dependent product. Notice also that we have not 
  crossed notation. Unit-purity implies that the general implication formula 
  applied to the former case recovers the identity exponential.
\end{rem}

\section{Cocartesianness \& Negation}
\label{section:cocartesian-negation}

The definition of a \emph{cocartesian double category} is an evident one on the 
pattern established by instantiating the abstract 2-categorical definition of a 
(co)cartesian object in the 2-category of double categories, pseudo double 
functors and tight transformations \cite{aleiferi2018}. It has been stated in an 
equivalent form and studied extensively from the standpoint of hypergraph 
categories and decorated cospans in \cite{patterson2023}. For our purposes, just 
as cartesianness leads to local products/conjunction $\wedge$ and local 
terminals/truth $\top$, cocartesianness leads to local disjunction $\vee$ and 
logical falsity/absurdity $\bot$.

\begin{defi}\label{def:cocartesian}
  A double category $\dbl D$ is \textbf{cocartesian} if the double functors 
    \begin{equation*}
      \Delta\colon\dbl D\to\dbl D\times\dbl D \qquad\text{and}\qquad !\colon \dbl D\to 1 
    \end{equation*}
  have left adjoints in the 2-category of double categories, pseudo double 
  functors and tight transformations denoted by $+$ and $\emptyset$, 
  respectively.
\end{defi}

\begin{rem}
  As observed in \cite{lambert2024a} for cartesian double categories, by duality 
  a cocartesian double category will have finite coproducts in each of the two 
  categories $\dbl D_0$ and $\dbl D_1$. These should be preserved appropriately 
  by the external structure functors of the double category. The tight 
  codiagonals and initials will be denoted 
    \begin{equation*}
      \nabla_x\colon x+x\to x\qquad\text{and}\qquad \init\colon\emptyset\to x
    \end{equation*}
  A cocartesian equipment will have local coproducts and initials in the 
  following way.
\end{rem}

\begin{lem} \label{lemma:local-coproducts}
  A cocartesian equipment has finite coproducts locally.
\end{lem}
\begin{proof}
  The construction is dual to that of local products. These are given by the 
  following extensions:
    \begin{equation*}
      \begin{tikzcd}
        {x+x} & {y+y} \\
        x & y
        \arrow[""{name=0, anchor=center, inner sep=0}, "{m+n}"{inner sep=.8ex}, "\shortmid"{marking}, from=1-1, to=1-2]
        \arrow["{\nabla_x}"', from=1-1, to=2-1]
        \arrow["{\nabla_y}", from=1-2, to=2-2]
        \arrow[""{name=1, anchor=center, inner sep=0}, "{m\vee n}"'{inner sep=.8ex}, "\shortmid"{marking}, from=2-1, to=2-2]
        \arrow["{\mathrm{ext}}"{description}, draw=none, from=0, to=1]
      \end{tikzcd}\qquad\qquad 
      \begin{tikzcd}
        \emptyset & \emptyset \\
        x & y
        \arrow[""{name=0, anchor=center, inner sep=0}, "{\id_\emptyset}"{inner sep=.8ex}, "\shortmid"{marking}, from=1-1, to=1-2]
        \arrow["{\init}"', from=1-1, to=2-1]
        \arrow["{\init}", from=1-2, to=2-2]
        \arrow[""{name=1, anchor=center, inner sep=0}, "\bot"'{inner sep=.8ex}, "\shortmid"{marking}, from=2-1, to=2-2]
        \arrow["{\mathrm{ext}}"{description}, draw=none, from=0, to=1]
      \end{tikzcd}
    \end{equation*}
  Coproduct injections to $m\vee n$ are induced from those into the coproduct 
  $m+n$. 
\end{proof}

\begin{defi} \label{def:negation-operator}
  For objects $x$ and $y$ in a locally cartesian closed equipment with initial 
  objects locally, the corresponding \textbf{negation operator} is the 
  implication operator 
    \begin{equation*}
      (-)\Rightarrow\bot\colon \dbl D(x,y) \to \dbl D(x,y) \qquad\qquad p\;\;\mapsto\;\; \neg p := p\Rightarrow \bot.
    \end{equation*}
  into the local initial $\bot$.
\end{defi}

\begin{exa} \label{example:negation}
  We end the section with some examples. In $\Rel$, negation is the usual 
  \emph{relational complement}. That is for $R\colon X\proto Y$, negation 
  $R\Rightarrow\bot$ is calculated as 
    \begin{equation*}
      R\Rightarrow\bot = \lbrace (x,y)\mid \neg (xRy)\rbrace
    \end{equation*}
  that is, the set of pairs \emph{not} related under $R$. As an example type of 
  query, consider the table 
    \begin{equation*}
      \begin{tabular}{| l | l | }
          \hline\multicolumn{2}{| c |}{\bf Inventory}\\
          \hline {\bf Apothecary }&{\bf Ingredient}\\
          \hline Arcadia's Cauldron & giant's toe \\
          \hline Arcadia's Cauldron & wheat \\
          \hline Arcadia's Cauldron & void salts \\
          \hline Elgrim's Elixers & giant's toe \\
          \hline Elgrim's Elixers & swamp fungal pod \\
          \hline The White Phial & Daedra heart \\
          \hline
      \end{tabular} 
    \end{equation*}
  instancing again the Inventory olog. Suppose that we are willing to work for 
  our wheat but we still need the giant's toe. Suppose also that we are deathly 
  allergic to swamp fungal pod. We will not even go in any shops carrying it. If 
  we just want to know which shops do not carry swamp fungal pod, we restrict 
  along the appropriate global element:
    
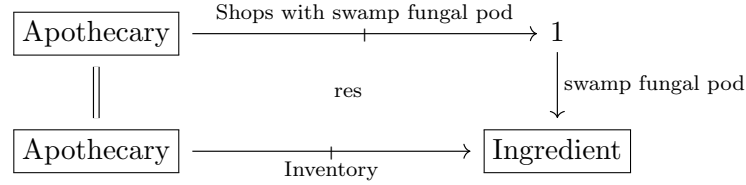
\begin{equation*}
      \begin{tikzcd}
        {\fbox{Apothecary}} &&&& 1 \\
        {\fbox{Apothecary}} &&&& {\fbox{Ingredient}}
        \arrow[""{name=0, anchor=center, inner sep=0}, "{\textrm{Shops with swamp fungal pod}}"{inner sep=.8ex}, "\shortmid"{marking}, from=1-1, to=1-5]
        \arrow[equals, from=1-1, to=2-1]
        \arrow["{\text{swamp fungal pod}}", from=1-5, to=2-5]
        \arrow[""{name=1, anchor=center, inner sep=0}, "{\mathrm{Inventory}}"'{inner sep=.8ex}, "\shortmid"{marking}, from=2-1, to=2-5]
        \arrow["{\mathrm{res}}"{description}, draw=none, from=0, to=1]
      \end{tikzcd}
    \end{equation*}
  then negate and restrict:
    
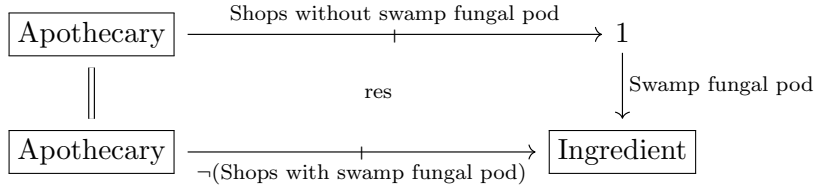
\begin{equation*}
      \begin{tikzcd}
        {\fbox{Apothecary}} &&&&& 1 \\
        {\fbox{Apothecary}} &&&&& {\fbox{Ingredient}}
        \arrow[""{name=0, anchor=center, inner sep=0}, "{\textrm{Shops without swamp fungal pod}}"{inner sep=.8ex}, "\shortmid"{marking}, from=1-1, to=1-6]
        \arrow[equals, from=1-1, to=2-1]
        \arrow["{\text{Swamp fungal pod}}", from=1-6, to=2-6]
        \arrow[""{name=1, anchor=center, inner sep=0}, "{\neg(\textrm{Shops with swamp fungal pod})}"'{inner sep=.8ex}, "\shortmid"{marking}, from=2-1, to=2-6]
        \arrow["{\mathrm{res}}"{description}, draw=none, from=0, to=1]
      \end{tikzcd}
    \end{equation*}
  returning the table 
    \begin{equation*}
      \begin{tabular}{| l |}
          \hline\multicolumn{1}{| c |}{\bf Shops without Swamp Fungal Pod}\\
          \hline Arcadia's Cauldron \\
          \hline The White Phial \\
          \hline
      \end{tabular}. 
    \end{equation*}
  Note that the negation proarrow as the codomain of the second restriction cell 
  above returns a collection of pairs, namely, those shop-ingredient pairs not 
  in the original relation. So, the second filter is actually for those shops 
  \emph{not related} to swamp fungal pod. In other words, we form the relation 
  returning those shops \emph{with} the ingredient, form the relational 
  complement of that, and then filter for those shops not carrying the 
  ingredient we want to avoid.

  Now, this is indeed a little elaborate. Alternatively, we can form a conjoint 
  and then negate and compose: 
    
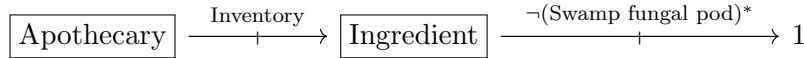
\begin{equation*}
      \begin{tikzcd}
        {\fbox{Apothecary}} && {\fbox{Ingredient}} &&&& 1
        \arrow["{\mathrm{Inventory}}"{inner sep=.8ex}, "\shortmid"{marking}, from=1-1, to=1-3]
        \arrow["{\neg(\text{Swamp fungal pod})^*}"{inner sep=.8ex}, "\shortmid"{marking}, from=1-3, to=1-7]
      \end{tikzcd}
    \end{equation*}
  which returns the same result in a simpler fashion but does not utilize in the 
  same way the native functional aspects of the olog. Note that the effect all 
  the way on the right is instanced by exactly all the ingredients other than 
  swamp fungal pod. The composite with its implied existential quantification 
  returns those shops having an ingredient on the list of everything but swamp 
  fungal pod. Now, we can combine this easily with a local conjunction to 
  execute the compound query where we look also for giant's toe:
    \begin{equation*}
      \resizebox{\textwidth}{!}
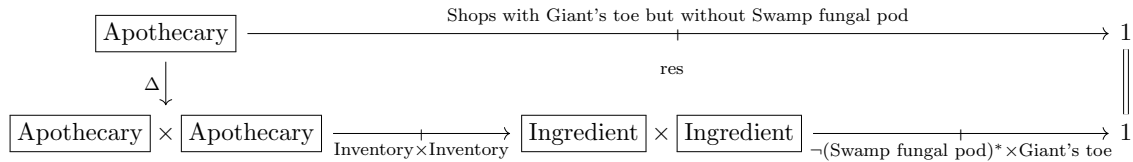
{%
        \begin{tikzcd}[ampersand replacement=\&]
        {\fbox{Apothecary}} \&\&\&\&\&\&\&\& 1 \\
        {\fbox{Apothecary}\times \fbox{Apothecary}} \&\&\& {\fbox{Ingredient}\times \fbox{Ingredient}} \&\&\&\&\& 1
        \arrow[""{name=0, anchor=center, inner sep=0}, "{\text{Shops with Giant's toe but without Swamp fungal pod}}"{inner sep=.8ex}, "\shortmid"{marking}, from=1-1, to=1-9]
        \arrow["\Delta"', from=1-1, to=2-1]
        \arrow[equals, from=1-9, to=2-9]
        \arrow["{\mathrm{Inventory}\times\mathrm{Inventory}}"'{inner sep=.8ex}, "\shortmid"{marking}, from=2-1, to=2-4]
        \arrow["{\neg(\text{Swamp fungal pod})^*\times\text{Giant's toe}}"'{inner sep=.8ex}, "\shortmid"{marking}, from=2-4, to=2-9]
        \arrow["{\mathrm{res}}"{description}, draw=none, from=0, to=2-4]
        \end{tikzcd}
        }
    \end{equation*}
  returning exactly the table
    \begin{equation*}
      \begin{tabular}{| l |}
          \hline\multicolumn{1}{| c |}{\bf Shops with Giant's Toe but without Swamp Fungal Pod}\\
          \hline Arcadia's Cauldron \\
          \hline
      \end{tabular}. 
    \end{equation*}
  of the single shop that has giant's toe and does not have swamp fungal pod. 
  Note that, as a result of the construction of local products, this process is 
  much like a promonoidal \emph{Day convolution} of the two effects. 
\end{exa}
\begin{exa} 
  This is also evident in a disjunction query, where for example we are happy to 
  patronize any shop having either of two desired ingredients:
    \begin{equation*}
      \resizebox{\textwidth}{!}{%
        \begin{tikzcd}[ampersand replacement=\&]
        {\fbox{Apothecary}+ \fbox{Apothecary}} \&\&\& {\fbox{Ingredient}\times \fbox{Ingredient}} \&\&\& 1 \\
        {\fbox{Apothecary}} \&\&\&\&\&\& 1
        \arrow["{\mathrm{Inventory}+\mathrm{Inventory}}"{inner sep=.8ex}, "\shortmid"{marking}, from=1-1, to=1-4]
        \arrow["\nabla"', from=1-1, to=2-1]
        \arrow["{\text{Wheat}+\text{Giant's toe}}"{inner sep=.8ex}, "\shortmid"{marking}, from=1-4, to=1-7]
        \arrow[""{name=0, anchor=center, inner sep=0}, "{\text{Shops with Wheat or Giant's toe}}"'{inner sep=.8ex}, "\shortmid"{marking}, from=2-1, to=2-7]
        \arrow[equals, from=2-7, to=1-7]
        \arrow["{\mathrm{ext}}"{description}, draw=none, from=1-4, to=0]
      \end{tikzcd}}
    \end{equation*}
  This is a local disjunction query, returning the table
    \begin{equation*}
      \begin{tabular}{| l |}
          \hline\multicolumn{1}{| c |}{\bf Shops with Wheat or Giant's Toe}\\
          \hline Arcadia's Cauldron \\
          \hline Elgrim's Elixers \\
          \hline
      \end{tabular}
    \end{equation*}
  of exactly those shops having either one but not necessarily both. We now 
  leave it to the reader's ingenuity to construct other combinations and 
  examples of interest.
\end{exa}

\section{Distributivity}
\label{section:distributivity}

Now that we have cocartesian structure, it is worth asking about its interaction 
with the cartesian structure. In particular, it is natural to ask for a 
\emph{distributive law}. Phrased for the two operations in the ambient 
2-category, it is a \emph{global} distributive law describing the interaction 
between the two operations in total. A main result of the present work is that 
global distributivity is stable under localization. That is, suppose we have an 
equipment that is both cartesian and cocartesian with Frobenius reciprocity for 
local products. If distributivity holds for this equipment as a cartesian and 
cocartesian object in the 2-category of double categories, then the induced 
local products and coproducts are distributive too. In this sense global 
distributivity localizes. The result requires some set up and notation.

\begin{con}[Global Distributor] \label{construction:global-distributor}
   Let $\dbl D$ denote a fixed equipment that is both cartesian and cocartesian. 
   Suppose that local products satisfy Frobenius. Now, induced via universal 
   properties are canonical comparison morphisms for both distributivity and 
   interchange: 
    \begin{equation*}
      \gamma \colon (x\times y)+(x\times z) \to x\times (y+z)\qquad\qquad
      \delta\colon (x\times y)+(x\times z) \to (x+x) \times (y+z)
    \end{equation*}
  Subscript indexing will be added as needed. These morphisms are of course 
  related at least by codiagonals:
    \begin{equation*}
      \delta(\nabla\times 1) = \gamma
    \end{equation*}
  as can be checked using the universal properties of the involved arrows. Now 
  additionally, $\delta$ commutes with diagonal and codiagonals in the sense 
  that 
    \begin{equation*}
      \delta(\Delta_x+\Delta_x) = \Delta_{x+x}\qquad\qquad (\nabla_x\times\nabla_x)\delta = \nabla_{x\times x}
    \end{equation*}
  both hold. The interchanger $\delta$ fits into the following commutative 
  diagram:
    \begin{equation*}
      \begin{tikzcd}
        {x+x} & x & {x\times x} \\
        {(x\times x)+(x\times x)} && {(x+x)\times (x+x)}
        \arrow["\nabla", from=1-1, to=1-2]
        \arrow["{\Delta+\Delta}"', from=1-1, to=2-1]
        \arrow["\Delta", from=1-2, to=1-3]
        \arrow["\delta"', from=2-1, to=2-3]
        \arrow["{\nabla\times\nabla}"', from=2-3, to=1-3]
      \end{tikzcd}
    \end{equation*}
  and in this sense each object $x$ is much like a bialgebra for the global rig 
  structure of $\dbl D$. Now, there is a canonical comparison cell of the form 
    \begin{equation*}
      \begin{tikzcd}
        {(x\times y)+(x\times z)} && {(x'\times y')+(x'\times z')} \\
        {x\times (y+z)} && {x'\times (y'+z')}
        \arrow[""{name=0, anchor=center, inner sep=0}, "{(m\times p)+(m\times q) }"{inner sep=.8ex}, "\shortmid"{marking}, from=1-1, to=1-3]
        \arrow["\gamma"', from=1-1, to=2-1]
        \arrow["\gamma", from=1-3, to=2-3]
        \arrow[""{name=1, anchor=center, inner sep=0}, "{m\times (p+q)}"'{inner sep=.8ex}, "\shortmid"{marking}, from=2-1, to=2-3]
        \arrow["\Gamma"{description}, draw=none, from=0, to=1]
      \end{tikzcd}
    \end{equation*}
  for any such triple of objects of $\dbl D$. This has associated to it globular 
  cells by \emph{sliding}, namely, 
    \begin{equation*}
      \begin{tikzcd}
        {(x\times y)+(x\times z)} && {(x'\times y')+(x'\times z')} & {x'\times (y'+z')} \\
        {(x\times y)+(x\times z)} & {x\times (y+z)} && {x'\times (y'+z')}
        \arrow[""{name=0, anchor=center, inner sep=0}, "{(m\times p)+(m\times q) }"{inner sep=.8ex}, "\shortmid"{marking}, from=1-1, to=1-3]
        \arrow[equals, from=1-1, to=2-1]
        \arrow["{\gamma_!}"{inner sep=.8ex}, "\shortmid"{marking}, from=1-3, to=1-4]
        \arrow[equals, from=1-4, to=2-4]
        \arrow["{\gamma_!}"'{inner sep=.8ex}, "\shortmid"{marking}, from=2-1, to=2-2]
        \arrow[""{name=1, anchor=center, inner sep=0}, "{m\times (p+q)}"'{inner sep=.8ex}, "\shortmid"{marking}, from=2-2, to=2-4]
        \arrow["{\Gamma_!}"{description}, draw=none, from=0, to=1]
      \end{tikzcd}
    \end{equation*}
  and 
    \begin{equation*}
      \begin{tikzcd}
        {x\times (y+z)} & {(x\times y)+(x\times z)} && {(x'\times y')+(x'\times z')} \\
        {x\times (y+z)} && {x'\times (y'+z')} & {(x'\times y')+(x'\times z')}
        \arrow["{\gamma^*}"{inner sep=.8ex}, "\shortmid"{marking}, from=1-1, to=1-2]
        \arrow[equals, from=1-1, to=2-1]
        \arrow[""{name=0, anchor=center, inner sep=0}, "{(m\times p)+(m\times q) }"{inner sep=.8ex}, "\shortmid"{marking}, from=1-2, to=1-4]
        \arrow[equals, from=1-4, to=2-4]
        \arrow[""{name=1, anchor=center, inner sep=0}, "{m\times (p+q)}"'{inner sep=.8ex}, "\shortmid"{marking}, from=2-1, to=2-3]
        \arrow["{\gamma^*}"'{inner sep=.8ex}, "\shortmid"{marking}, from=2-3, to=2-4]
        \arrow["{\Gamma^*}"{description}, draw=none, from=0, to=1]
      \end{tikzcd}
    \end{equation*}
  each of which is invertible if the original $\Gamma$ is invertible 
  \cite[Lemma A.7]{patterson2026}. These are called the \emph{commutor} and 
  \emph{cocommutor} cells associated to $\Gamma$. The data of the arrows 
  $\gamma$ and cells $\Gamma$ organize into a tight transformation of double 
  categories 
    \begin{equation*}
      \begin{tikzcd}
        {\dbl D^{\times3}} &&& {\dbl D^{\times2}} \\
        {\dbl D^{\times4}} \\
        {\dbl D^{\times4}} && {\dbl D^{\times2}} & {\dbl D}
        \arrow["{1\times (-+=)}", from=1-1, to=1-4]
        \arrow["{\Delta\times 1\times 1}"', from=1-1, to=2-1]
        \arrow["{\times }", from=1-4, to=3-4]
        \arrow["{1\times \tau\times 1}"', from=2-1, to=3-1]
        \arrow["\Gamma"{description}, between={0.2}{0.8}, Rightarrow, from=3-1, to=1-4]
        \arrow["{(-+=)\times (?+??)}"', from=3-1, to=3-3]
        \arrow["\times"', from=3-3, to=3-4]
      \end{tikzcd}
    \end{equation*}
  where $\tau$ is the canonical morphism interchanging the factors of the 
  cartesian product. The (co)cartesian structure is said to be \textbf{globally 
  distributive} if the arrow and cell components of $\Gamma$ viewed as a 
  transformation are all invertible. Under this assumption of invertibility note 
  that we thus have an invertible cell 
    \begin{equation*} 
      \resizebox{\textwidth}{!}{%
        \begin{tikzcd}[ampersand replacement=\&]
        {x\times (y+z)} \& {(x\times y)+(x\times z)} \&\& {(x'\times y')+(x'\times z')} \& {x'\times (y'+z')} \\
        {x\times (y+z)} \&\& {x'\times (y'+z')} \& {(x'\times y')+(x'\times z')} \& {x'\times (y'+z')} \\
        {x\times (y+z)} \&\& {x'\times (y'+z')} \&\& {x'\times (y'+z')}
        \arrow["{\gamma^*}"{inner sep=.8ex}, "\shortmid"{marking}, from=1-1, to=1-2]
        \arrow[equals, from=1-1, to=2-1]
        \arrow[""{name=0, anchor=center, inner sep=0}, "{(m\times p)+(m\times q) }"{inner sep=.8ex}, "\shortmid"{marking}, from=1-2, to=1-4]
        \arrow[""{name=1, anchor=center, inner sep=0}, "{\gamma_!}"{inner sep=.8ex}, "\shortmid"{marking}, from=1-4, to=1-5]
        \arrow[equals, from=1-4, to=2-4]
        \arrow[equals, from=1-5, to=2-5]
        \arrow[""{name=2, anchor=center, inner sep=0}, "{m\times (p+q)}"'{inner sep=.8ex}, "\shortmid"{marking}, from=2-1, to=2-3]
        \arrow[from=2-1, to=3-1]
        \arrow["{\gamma^*}"'{inner sep=.8ex}, "\shortmid"{marking}, from=2-3, to=2-4]
        \arrow[equals, from=2-3, to=3-3]
        \arrow[""{name=3, anchor=center, inner sep=0}, "{\gamma_!}"'{inner sep=.8ex}, "\shortmid"{marking}, from=2-4, to=2-5]
        \arrow[equals, from=2-5, to=3-5]
        \arrow[""{name=4, anchor=center, inner sep=0}, "{m\times (p+q)}"', from=3-1, to=3-3]
        \arrow[""{name=5, anchor=center, inner sep=0}, "\id"'{inner sep=.8ex}, "\shortmid"{marking}, from=3-3, to=3-5]
        \arrow["{\Gamma^*}"{description}, draw=none, from=0, to=2]
        \arrow["1"{description}, draw=none, from=1, to=3]
        \arrow["1"{description, pos=0.6}, draw=none, from=2, to=4]
        \arrow["\cong"{description, pos=0.4}, draw=none, from=2-4, to=5]
      \end{tikzcd}}
    \end{equation*}
  since $\gamma$, being invertible, induces a proarrow adjoint equivalence. This 
  cell of course collapses to an identity if $\dbl D$ is also locally posetal. 
  Thus, in the loose bicategory of $\dbl {D}$, proarrow distributivity takes the 
  form of what we think of as a \emph{conjugacy rule}. This will be used in the 
  calculation of the following proposition, the main result of the section.
\end{con}

\begin{thm} \label{theorem:local-distributivity}
  Let $\dbl D$ denote a cocartesian `double category of relations'. If $\dbl D$ 
  is globally distributive, then local products distribute over local coproducts 
  in the sense that
    \begin{equation*}
      (m\wedge p)\vee (m\wedge q) = m\wedge (p\vee q)
    \end{equation*}
  holds in $\dbl D(x,y)$. In this sense $\dbl D$ enjoys \textbf{local 
  distributivity}.
\end{thm}
\begin{proof}
  The argument is inspired by the reverse direction of 
  \cite[Proposition 9.2.3]{jacobs1999} inasmuch as Frobenius reciprocity is used 
  to bring the companions and conjoints forming the conjunction as an extension 
  cell outside the local product to let the global distributivity conjugacy isos 
  do the rest of the work. Throughout we use the notation and terminology of 
  \cref{construction:global-distributor} above. In detail, first observe that by 
  Frobenius reciprocity and the construction of local products and coproducts by 
  restrictions and extensions:
    \begin{align*}
      m\wedge(p\vee q) &= (m\wedge\coprod_{\nabla,\nabla}(p+q))\\
      &=\coprod_{\nabla,\nabla}((\nabla,\nabla)^*(m)\wedge(p+q))\\
      &= \nabla^*\odot((\nabla_!\odot m \odot \nabla^*)\wedge (p+q)) \odot \nabla_!\\
      &= \nabla^*\odot(\Delta_!\odot ((\nabla_!\odot m \odot \nabla^*)\times (p+q))\odot\Delta^*) \odot \nabla_!.
    \end{align*}
  Note that we are suppressing the indexing subscripts. Now, working from the 
  last line, we can pull the most imbedded codiagonals outside the product, and 
  then use distributivity as proarrow conjugacy:
    \begin{align*}
      &\nabla^*\odot(\Delta_!\odot ((\nabla_!\odot m \odot \nabla^*)\times (p+q))=\Delta^*) \odot \nabla_! \\
      = \;&\nabla^*\odot\Delta_!\odot(\nabla_!\times 1)\odot (m\times (p+q))\odot (\nabla^*\times 1)\odot \Delta^*\odot\nabla_!\\
      =\;&\nabla^*\odot\Delta_!\odot(\nabla_!\times 1)\odot \gamma^*\odot ((m\times p)+ (m+q))\odot \gamma_!\odot (\nabla^*\times 1)\odot \Delta^*\odot\nabla_!.
    \end{align*}
  Now, from the last line, we use the relationship between $\gamma$ and the 
  interchanger $\delta$, to get 
    \begin{align*}
      &\nabla^*\odot\Delta_!\odot(\nabla_!\times 1)\odot \gamma^*\odot ((m\times p)+ (m+q))\odot \gamma_!\odot (\nabla^*\times 1)\odot \Delta^*\odot\nabla_!\\
      =\;&\nabla^*\odot\Delta_!\odot\delta^*\odot ((m\times p)+ (m+q))\odot \delta_!\odot \Delta^*\odot\nabla_!.
    \end{align*}
  But the interchanger commutes with the diagonals and thus commutes with their 
  companions and conjoints since, being invertible, it induces a proarrow 
  adjoint equivalence. So, the last line immediately above is calculated to give 
  the desired disjunction of conjunctions:
    \begin{align*}
      &\nabla^*\odot\Delta_!\odot\delta^*\odot ((m\times p)+ (m+q))\odot \delta_!\odot \Delta^*\odot\nabla_!\\
      =\;&\nabla^*\odot(\Delta_!+\Delta_!) \odot ((m\times p)+ (m+q))\odot (\Delta^*+\Delta^*)\odot\nabla_!\\
      =\;&\nabla^* \odot (\Delta_!\odot (m\times p)\odot\Delta^*+ \Delta_!\odot(m+q)\odot\Delta^*)\odot\nabla_!\\
      =\;&\nabla^* \odot ((m\wedge p) + (m\wedge q))\odot\nabla_!\\
      =\;&(m\wedge p) \vee (m\wedge q).
    \end{align*}
  Stringing all the calculations together shows that distributivity does hold in 
  $\dbl D(x,y)$, as required.
\end{proof}

\begin{rem}
  The argument above can actually be carried out in the case where $\dbl D$ is 
  merely a cartesian and cocartesian equipment satisfying Frobenius reciprocity 
  and global distributivity. The equalities above are all replaced by 
  isomorphisms and tracking the typing shows that they are all globular. 
  Nonetheless, we phrase the argument in terms of `double categories of 
  relations' since that is the presented forum for double ologs, and required is 
  a fuller account of the relationship between general loose compactness, 
  Frobenius reciprocity, and distributivity, which we leave to a future 
  investigation. Note that we have not assumed anything about the `double 
  category of relations' having any dependent products or being cartesian closed.
\end{rem}

\begin{cor}
  Any cocartesian `double category of relations' whose coproducts satisfy 
  Beck-Chevalley and with global distributivity in the above sense is a 
  \emph{coherent fibration} in the sense of \cite[Definition 4.2.1]{jacobs1999}.
\end{cor}
\begin{proof}
  This is immediate from \cref{theorem:local-distributivity} and the definition 
  of a coherent fibration.
\end{proof}

\section{Modal FOL \& Description Logic}
\label{section:FOL-DescriptionLogic}

Combining the structures encountered so far we arrive at the richest 
double-categorical forum for data manipulation and interpretation of logic. This 
is a double olog with dependent products, comprehension and distributive 
coproducts. Such a structure is rich enough to interpret first-order predicate 
logic when viewed as an equipment and satisfies several canonical modal axioms.

\begin{defi}
  A \textbf{FOML double olog} is a double olog with dependent products, strong 
  tabulators, and cocartesian products satisfying global distributivity.
\end{defi}

Since this is a rather heavy definition, we summarize the developed structures 
and properties, namely, that, a FOML double olog 
  \begin{enumerate}
    \item is a $\prod$-double olog: it has first-order quantification satisfying 
    Beck-Chevalley (\cref{def:prod-double-olog,lemma:beck-chevalley});
    \item has local products and equality satisfying Frobenius reciprocity 
    (\cref{prop:Frobenius-reciprocity});
    \item is locally cartesian closed (\cref{theo:locally-cartesian-closed});
    \item has local disjunction and negation (\cref{lemma:local-coproducts});
    \item is locally distributive (\cref{theorem:local-distributivity});
    \item and interprets modal necessity and possibility 
    (\cref{def:necessity-from-dep-products}).
  \end{enumerate}
This, as has been showcased throughout the development, is more than enough 
structure to handle classical relational algebra and the database queries 
developed in \cite{codd1972} as well as modal necessity (safety) and possiblity 
(liveness). Putting this in the context of fibrational semantics, we have the 
following culminating result.

\begin{cor} \label{corollary:first-order-fibration}
  A FOML double olog is a first-order fibration 
  \cite[Definition 4.2.1]{jacobs1999}.
\end{cor} 
\begin{proof}
  Certianly the source-target projection is a coherent fibration, as observed 
  above. Local cartesian closure and all dependent products are all that is 
  additionally required.
\end{proof} 

As a consequence, any FOML double olog $\dbl D$, viewed as a fibration 
$\dbl D_1\to\dbl D_0\times\dbl D_0$, interprets first-order predicate logic in 
the manner described in \cite[\S 4.3]{jacobs1999}. This result is thus one 
realization of the general proposal and project of viewing structured equipments 
through the lense of fibrational semantics \cite{lambert2025}. More development 
on this point will be left to future work (see 
\cref{section:conclusion-compactness}). Now, with regard to the interaction 
between modal operators and propositional connectives we have the following. The 
second is the standard (K) axiom assumed of any sensible modal system.

\begin{prop}
  In any FOML double olog, the inequalities 
    \begin{enumerate}
      \item $\Box\phi\wedge\Box\psi\leq \Box(\phi\wedge\psi)$
      \item $\Box(\phi\Rightarrow\psi)\leq \Box\phi\Rightarrow\Box\psi$
    \end{enumerate}
  both hold.
\end{prop}
\begin{proof}
  For the first, $\Box$ is a right adjoint in any case and thus preserves 
  products. Secondly, we have the derivation  
    \begin{prooftree}
      \AxiomC{$\phi\Rightarrow\psi \leq \phi\Rightarrow\psi$}
      \UnaryInfC{$\phi\Rightarrow\psi\wedge\phi \leq \psi$}
      \UnaryInfC{$\Box(\phi\Rightarrow \psi \wedge\phi) \leq \Box\psi$}
      \UnaryInfC{$\Box(\phi\Rightarrow \psi) \wedge\Box\phi \leq \Box\psi$}
      \UnaryInfC{$\Box(\phi\Rightarrow \psi) \leq \Box\phi\Rightarrow\Box\psi$}
    \end{prooftree}
  using cartesian closedness, product preservation, and the fact that $\Box$ is 
  a functor.
\end{proof}

Having developed the previous structures, we note in closing that FOML double 
ologs are rich enough to interpret description logic 
\cite{brachman2004,baader2010}. This is the basis of web ontology languages 
(OWL) used in the development of web3 and a form of knowledge representation. 
The standard definition of the logic and its set-theoretic semantics is the 
following.

\begin{defi}
    An \textbf{attributive concept language} (ACL) is given by a signature 
    $\mathcal A = (O,A,R)$, consisting of objects $O$, atomic concepts $A$ and 
    roles $R$, and a set of concepts, defined as the smallest set closed under 
    the following formation rules:
        \begin{enumerate}
            \item $\top$, $\bot$ are concepts;
            \item each atomic concept is a concept;
            \item if $C$ is a concept, so is $\neg C$;
            \item if $C$ and $D$ are concepts, so are $C\cap D$ and $C\cup D$;
            \item if $C$ is a concept and $R$ is a role, then $\forall R.C$ and $\exists R. C$ are concepts.
        \end{enumerate}
    An \textbf{interpretation} of an ACL consists of a set $X$ and a 
    correspondence $\llbracket-\rrbracket$ assigning 
        \begin{enumerate}
            \item a set element $\llbracket a\rrbracket \in X$ for each object $a\in O$;
            \item to each concept $C$ a subset $\llbracket C\rrbracket\subset X$;
            \item to each role $R$ a binary relation $\llbracket R\rrbracket \subset X\times X$;
        \end{enumerate}
    in such a way that  
        \begin{enumerate}
            \item $\llbracket C\cap D\rrbracket = \llbracket C\rrbracket \cap \llbracket D\rrbracket$
            \item $\llbracket C\cup D\rrbracket = \llbracket C\rrbracket \cup \llbracket D\rrbracket$
            \item $\llbracket \neg C \rrbracket = X\setminus \llbracket C\rrbracket$
            \item $\llbracket\exists R.C\rrbracket = \lbrace x\in X\mid \text{there is } y\in X \text{ such that } (x,y)\in R \text{ and } y\in C\rbrace$ 
            \item $\llbracket\forall R.C\rrbracket = \lbrace x\in X\mid \text{for all } y\in X \text{ if } (x,y)\in R \text{ then } y\in C\rbrace$. 
        \end{enumerate}
  Denote such an interpretation by $\mathcal A \to \Rel$.
\end{defi}

The notation is chosen because, technically, we are using the double-category 
data of relations without also using the full double-structure. This of course 
suggests a more general interpretation in any suitably structured `double 
category of relations'. Note that the interpretation of quantifiers is exactly 
the modal statements relative to the fixed subset/proposition summarized in 
\cref{fig:quant-queries}.

\begin{defi} \label{definition:interpret-description-logic}
  An \textbf{interpretation} of an ACL in an FOML double olog $\dbl D$ consists 
  of an object $x$ and a correspondence $\llbracket-\rrbracket$ assigning 
    \begin{enumerate}
        \item a term $\llbracket a\rrbracket \colon u_a\to x$ for each object $a\in O$;
        \item to each concept $C$ a monic $\llbracket C\rrbracket\rightarrowtail x$;
        \item to each role $R$ a proarrow $\llbracket R\rrbracket \colon x\proto x$;
    \end{enumerate}
  in such a way that  
    \begin{enumerate}
        \item $\llbracket C\cap D\rrbracket = \llbracket C\rrbracket \wedge \llbracket D\rrbracket$
        \item $\llbracket C\cup D\rrbracket = \llbracket C\rrbracket \vee \llbracket D\rrbracket$
        \item $\llbracket \neg C \rrbracket = \llbracket C\rrbracket\Rightarrow \bot$
        \item $\llbracket\exists R.C\rrbracket = \diamondsuit_R\llbracket C\rrbracket$ 
        \item $\llbracket\forall R.C\rrbracket = \Box_R\llbracket C\rrbracket$.
    \end{enumerate}
  Denote such an interpretation by $\mathcal A \to \dbl D$.
\end{defi}

\begin{rem}
  Note that the modal operators make the interpretation of the quantified concepts over the roles very straightforward. And again this situation of having a relation and a proposition/concept on either the source or target of the relation is exactly our Inventory olog together with the shopping list that began the paper in \cref{section-ologs}.
\end{rem}

\section{Conclusion \& Prospectus: The Role of Compactness}
\label{section:conclusion-compactness}

It has been evident to the author at least since the work on `double categories 
of relations' that a missing account of \emph{loose compactness} is an profound 
bottleneck on further development of the present theory (see the discussion of 
\cite[\S 12]{lambert2022}). In fact, this concern goes back to a first exposure 
to Yoneda structures in \cite{weber2007} where a form of compactness is an 
essential ingredient in phrasing what it means to be a \emph{2-topos}. Owing to 
the need of recasting this derived loose structure with a first-class status, 
the author has had in mind a prospective synthetic axiomatization of profunctors 
as a double category leading to a notion of a \emph{double topos}. Loose 
compactness would thus naturally be a central feature. Developments meant to 
supply this lack in the double-categorical literature have since been initiated 
in related work \cite{patterson2024a,patterson2024b}. Especially the recent 
fruitful collaboration on \emph{twisted double functors} \cite{lambert2026} is 
meant to supply a precise framework for a missing notion of a \emph{twisted 
adjunction} and a \emph{loose Yoneda theory} (orthogonal to that of 
\cite{pare2011}) in which a general account of loose compactness can be phrased. 
The point of this concluding section is to discuss the necessity of such a 
framework and the horizons its successful articulation would open. 

The necessity of a general account of double-categorical loose compactness is 
implicit in the original \cite{carboni1987} and glaring in \cite{lambert2022}. 
`Bicategories of relations' and the more general \emph{cartesian bicategories} 
of the former reference are an alternative to \emph{allegories} \cite{freyd1990} 
presenting an axiomatization of the \emph{loose bicategorical} structure of 
relational systems. As powerful as the framework is, however, there are two 
subtleties that led to the the work of \cite{lambert2022}. The first is the fact 
that the purely bicategorical approach, in an ironic twist relative to purely 
functional ologs, leaves out genuinely functional arrows as first-class 
entities, rather encoding them as certain special relations, namely, the 
\emph{maps}. As we have already noted, this makes naming, subtyping, and 
equational reasoning unnecessarily difficult; and the introduction of double 
categories elegantly solves this problem. A more fundamental issue with the 
framing of the approach, however, is that a `bicategory of relations' is asked 
to be \emph{locally posetal}. This indeed has the effect of simplifying many 
calculations (the derivation of the modular laws and compactness), but also more 
fundamentally of essentially reorienting the whole framework to that of a 
\emph{monoidal 2-category} since the bicategorical coherence conditions as a 
result all hold on the nose.

Now, from the standpoint of double category theory, this assumption thus has the 
effect of transposing the entire theory from the loose bicategorical direction 
to the tight 2-categorical direction. Of course, in essence, the entire point of 
double categories is to keep these as largely independent but related orthogonal 
directions axiomatizing first-class entities. This is why the equipment 
structure is necessary in \cite{lambert2022} to show that the loose bicategory 
of a `double category of relations' is a `bicategory of relations'. For the 
axioms governing diagonals in the cartesian structure \cite{aleiferi2018} are 
posed in the \emph{tight direction} and thus have to be transposed via the 
equipment structure to the loose bicategory. The happy circumstance is that the 
right adjoints are provided by conjoints; in other words, this generalization is 
impossible without the automatic symmetry of the full equipment structure. This 
is not to say that the definition of a `bicategory of relations' is the wrong 
one; it is just that it is not really about bicategories. The proper 
generalization should actually be closer to \emph{the underlying monoidal 
2-category of a cartesian double category (with suitable further structure) is a 
`bicategory of relations'}. This explains the considerable effort involved in 
the follow up paper \cite{carboni2008} which removes the assumption of 
\emph{locally posetal}. Essentially, in summary, this follow-up paper 
inadvertantly conflates categorification and transposition. The work on doing 
this transposition carefully has appeared only recently in \cite{patterson2026}.

Now, the present work follows \cite{lambert2025} in treating `double categories 
of relations' as a basic structural framework for double ologs. This of course 
carries the assumption again that each such olog is locally posetal. This means 
that from the fibrational perspective, our double ologs are models of a 
\emph{logic over a type theory} rather than of a \emph{depedent type theory} 
where there may be more than one morphism between given proarrows. The 
development of the present work shows that the operative structures in 
double-olog-ing and querying are really minimally (1) (co)cartesian structure, 
(2) equipment structure with Beck-Chevalley and for maximal expressivity also 
(3) dependent products, (4) strong tabulators, and the derived rules of 
modularity, Frobenius reciprocity and compactness. Now, again, these derived 
rules are consequences of the `double category of relations' structure and 
certainly depend for their proofs and form on discreteness and the locally 
posetal assumption. Now, discreteness is still of interest even if locally 
posetal is dropped. This is like moving from equivalence relations to general 
groupoids. In particular, the triviality of the duality involution is specific 
to relations and such structures. So, on the one hand, the locally posetal 
assumption is a computational convenience and a simplification making a 
direction connection between databases and logics over type theories (extending 
the reflections and program of \cite[\S 2.5]{Spivak2010}). On the other hand, it 
seems like an artificial limitation on the databases that can be modeled, and on 
the full connection with dependent type theory that is lurking in the 
background, inasmuch as the essential features can just be axiomatized: a 
distributive (co)cartesian equipment with strong tabulators and some suitably 
weakened versions of Frobenius reciprocity and of Beck-Chevalley for dependent 
(co)products. 

Whether compactness is axiomatized directly or derived in this more general 
framework, it presents a problem in that it is not really known what to look for 
without the locally posetal assumption. In particular, the general case should 
allow a non-trivial duality involution modeled by \emph{opposite} in 
profunctors, or perhaps by \emph{reverse path} in directed homotopy theory 
\cite{grandis2003,grandis2002directed2,grandis2009directedBook}. And, in 
essence, although it has been studied for 2-categories 
\cite{shulman2018contravariance}, there is no precise \emph{general theory} 
recovering oppositization and its interaction with monoidal products in the 
loose direction for double categories. This, again, on our view, requires the 
framework of loose adjunctions and a loose Yoneda theory stemming from the 
recent machinery of twisted double functors and in particular the twisted hom 
functor \cite{lambert2026}. Points of inspiration for such a theory are of 
course the ordinary 1-categorical accounts of *-autonomy 
\cite{barr1999autonomous} and compact closedness \cite{kelly1980coherence} but 
also the bicategorical approaches to duality and involutions in 
\cite{daystreet1997,shulman2018contravariance}.

Owing to the locally posetal assumption and the resulting triviality of the 
accompanying duality involution, compactness in the case of `double categories 
of relations' seems an apparent triviality. And indeed it could be observed that 
the bookkeeping problem of remembering which objects are on the left or right 
side of a relation is merely a nuisance that \emph{ought} to be remedied by 
something like compactness. It is maintained here, however, that this view is 
wrong, or at least short-sighted, on at least two counts. The first is of course 
already apparent in the nature of double ologs. Typically, for example, we 
naively think of a relation such as \emph{parent of} as pointing from parents to 
children, and not the other way around, much less from, or to, some product of 
the two types. Likewise, as a technical matter, relational division in the 
original \cite{codd1972} is actually framed for two tables representing 
relations of essentially arbitrary arity. Compactness is in the background 
grouping by moving columns from one side to the other to make sure the division 
columns match without incorporating those not involved in the query. Compactness 
plays non-trivially in both cases; in a salutory way in the latter by making 
sure the division works properly; but in an artificial way in the former 
inasmuch as it collapses intended semantic distinctions. In either case, 
compactness is doing non-trivial work and needs to be better understood.

Now, the second count is of course that this proposed view refuses to see this 
simpler case as a testing ground for the generalized theory of compactness with 
non-trivial duality as discussed above. This general account reveals an endemic 
\emph{sidedness} in that the example of profunctors is presented with opposite 
on one or the other side. This is necessary also for Lawvere metric spaces 
\cite{lawvere2002}, due to lack of symmetry, and for order ideals 
\cite[\S 3.4.6]{grandis2019} and other enriched profunctors. Now, our convention 
is that profunctors are typed with opposite on the left. This preserves the 
typing of the hom functors in diagrammatic order. Thus, profunctors from the 
terminal (for us anyway) are essentially copresheaves and profunctors to the 
terminal are presheaves. Now, the former are algebraic objects and the latter 
are geometric. We liken the former to \emph{states} and the latter to 
\emph{effects} or \emph{generalized propositions} as in CQM 
\cite{AbramskyCoecke2004,Selinger2007,HeunenVicary2019}. Or again the former to 
points or certain spaces and the latter to measurements or quantities; in this 
sense, the adjoint situation of \emph{Isbell conjugacy} 
\cite[\S 7]{Lawvere2005Reprint} applied to profunctors is already a high-level 
statement of space-quantity duality and thus also state-effect duality. In this 
way, we see from a very broad philosophical perspective that although 
compactness with trivial duality seems to conflate states and effects, it really 
ought not to in a general account. And arguably this ought not to be thought of 
as a collapse even in the trivial case. For in the example of matrices, states 
are column vectors and effects are row vectors. If these are set-indexed, again 
we have only a trivial duality involution, and although row and column vectors 
are related by transposition, surely no one would simply conflate them as 
\emph{the same} or argue that it does not matter how they are arranged. It seems 
to us worthwhile, then, even in the trivial case, to preserve a semantic 
distinction even if the bijections say we can move freely between them without 
dualizing objects/types explicitly.

Moreover, this two-sidedness of \emph{split contexts} for relational or 
propositional typing in general double ologs is consistent with a broader view 
of the semantics that especially enriched profunctors already do to provide. For 
example, the connection between *-autonomy, enriched structures and models of 
linear logic is well-established \cite{seely1989,barr1991,day2004}. Likewise, 
profunctors have explicitly and directly been studied as models of linear logic 
\cite{dunn2015}. \emph{Chu spaces} \cite[Appendix]{barr1979} naturally organize 
into a *-autonomous category. These all have such split contexts; in paricular, 
we view each Chu space as a scalar-valued profunctor-like morphism from the 
product of an underlying space and its frame of opens. That is, in summary, 
profunctor-type models of linear logic already incorporate some kind of 
non-trivial duality and an inherent sidedness in variable contexts. Our view is 
that the fact that the requisite care in the treatment of split contexts 
appearing as an inevitability in a generalization of locally posetal ologs is 
not a coincidence. For these locally posetal ologs are local discretizations of 
more general cartesian equipment-type structures that are generally 
profunctor-type double categories. And the references above suggest the types of 
logics and type theories these structures are meant to interpret.

From a type-theoretic perspective, the generalization is represented as a move 
from logics over type theories to some \emph{dependent type theory} where 
proarrows, formerly identified as relational propositions in a split context, 
are now identified as dependent types in a \emph{split context} whose two sides 
are mediated by a type-theoretic instantiation of compactness. Sidedness is the 
apparently novel feature. But the technical point is to axiomatize and provide 
double-categorical semantics for something like \emph{opposite types} 
\cite{Zeilberger2009,AgudeloAgudelo2022,AgudeloAgudelo2022a,rivera2026} 
resulting in a \emph{directed dependent type theory with split contexts} as the 
internal language of such structures. And in fact oppositization has already 
been likened to a way to transition between \emph{polarities} from 
\emph{inductive} to \emph{coinductive} types, or from \emph{positive} to 
\emph{negative} types, where the former has rules for constructing terms and the 
latter has rules for deconstructing terms \cite{levy2004cbpv,levy2005adjunction,Zeilberger2009,LicataZeilbergerHarper2008,LicataHarper2009,ahman2016}. 
This is much like the \emph{prover-denier duality} of game semantics which 
indeed brings us full-circle back to Frobenius algebras and Frobenius pairs 
\cite{mellies2012,mellies2016} (i.e. essentially generalized discreteness in the 
language of \cite{carboni1987} except again with non-trivial duality). 
Naturally, it is no coincidence that Chu spaces have served as models of 
concurrency and game semantics \cite{pratt1995}. This is all to say that a 
successful elementary axiomatization of profunctors as a double topos (as 
prognosticated in our \cite[\S 12.3]{lambert2022}) has an immense payoff, 
naturally providing semantic models of fragments of linear logic over directed 
type theories and providing a common framework for discussion of data, programs, 
concurrency, game semantics, and ultimately, we believe, CQM, synthetic 
probability theory \cite{fritz2020} and the synthetic branching spacetime models 
discussed in \cite{lambert2024b}.

\paragraph{Acknowledgements.} Early versions of this work were presented at the 
NU Mathematics Colloquium and at FMCS 2026. Thanks are due to all the 
participants of both meetings and particularly to Marcello LanFranchi, Dorette 
Pronk, and Peter Selinger for several provocative questions. Special thanks to 
Evan Patterson for reviewing an early draft of the manuscript; and to Darien 
DeWolf for the invitation to speak at FMCS 2026 which prompted much of the work 
in this paper.

\bibliographystyle{alphaurl}   
\bibliography{double-data-bib} 

\end{document}